\documentclass[a4paper, 11pt]{amsart}

\usepackage{a4wide}
\usepackage[english]{babel} 
\usepackage{amsmath, amssymb} 
\usepackage{tikz}
\usepackage{tikz-cd}
	\tikzcdset{row sep/normal=1.5em}
	\tikzcdset{column sep/normal=1.5em}
\usepackage{todonotes}

\usepackage[mathscr]{euscript}

\usepackage[normalem]{ulem}

\usepackage{hyperref}

\newcommand{\Q}{\mathbb{Q}}

\newcommand{\N}{\mathbb{N}}
\newcommand{\Z}{\mathbb{Z}}
\newcommand{\F}{\mathbb{F}}

\renewcommand{\P}{\mathbb{P}}
\newcommand{\A}{\mathbb{A}}

\newcommand{\T}{\mathbb{T}}

\renewcommand{\L}{\mathrm{L}}

\newcommand{\cA}{\mathcal{A}}
\newcommand{\cB}{\mathcal{B}}

\newcommand{\cT}{\mathcal{T}}

\newcommand{\cC}{\mathcal{C}}

\newcommand{\BMod}{\mathrm{BMod}}

\newcommand{\Alg}{\mathrm{Alg}}

\newcommand{\cQ}{\mathcal{Q}}
\newcommand{\cR}{\mathcal{R}}

\newcommand{\E}{\mathbb{E}}

\newcommand{\can}{\mathrm{can}}

\newcommand{\op}{\mathrm{op}}
\newcommand{\Cat}{\mathrm{Cat}}
\renewcommand{\Pr}{\mathrm{Pr}}

\newcommand{\perf}{\mathrm{perf}}
\newcommand{\st}{\mathrm{st}}

\renewcommand{\SS}{\mathbb{S}}
\newcommand{\Mon}{\mathrm{Mon}}
\newcommand{\Grp}{\mathrm{Grp}}
\renewcommand{\E}{\mathbb{E}}
\newcommand{\Spc}{\mathrm{Spc}}
\newcommand{\Set}{\mathrm{Set}}

\newcommand{\lto}{\longrightarrow}

\newcommand{\mmod}{/\!\!/}
\newcommand{\sslash}{\mmod}

\newcommand{\id}{\mathrm{id}}

\newcommand{\pr}{\mathrm{pr}}

\DeclareMathOperator{\Sp}{Sp}

\newcommand{\laxtimes}[1]{\mathop{\times\mkern-13mu\raise1.3ex\hbox{$\scriptscriptstyle\to$}_{#1}}}
\newcommand{\raxtimes}[1]{\mathop{\times\mkern-13mu\raise1.3ex\hbox{$\scriptscriptstyle\leftarrow$}_{#1}}}
\newcommand{\slax}{ \times\mkern-14mu\raise1ex\hbox{$\scriptscriptstyle\to$} }

\DeclareMathOperator{\RMod}{RMod}
\DeclareMathOperator{\Mod}{Mod}
\DeclareMathOperator{\Perf}{Perf}
\DeclareMathOperator{\Proj}{Proj}
\DeclareMathOperator{\Tor}{Tor}
\DeclareMathOperator{\Map}{Map}
\DeclareMathOperator{\map}{map}
\DeclareMathOperator{\End}{End}
\DeclareMathOperator{\Fun}{Fun}

\DeclareMathOperator{\Ind}{Ind}
\DeclareMathOperator{\fib}{fib}

\DeclareMathOperator{\cof}{cofib}
\DeclareMathOperator{\TC}{TC}

\DeclareMathOperator{\THH}{THH}
\DeclareMathOperator{\HC}{HC}
\DeclareMathOperator{\HH}{HH}

\DeclareMathOperator{\HP}{HP}
\DeclareMathOperator{\Nil}{Nil}

\newcommand{\wtimes}[2]{\odot_{#1}^{#2}}

\newtheorem{thm}{Theorem}
\newtheorem*{thm*}{Theorem}
\newtheorem{cor}[thm]{Corollary}
\newtheorem*{cor*}{Corollary}
\newtheorem{lemma}[thm]{Lemma}
\newtheorem{prop}[thm]{Proposition}

\theoremstyle{definition}
\newtheorem{dfn}[thm]{Definition}
\newtheorem{construction}[thm]{Construction}
\newtheorem{ex}[thm]{Example}
\newtheorem{rem}[thm]{Remark}
\newtheorem*{dfn*}{Definition}
\newtheorem*{Notation}{Notation}
\newtheorem*{introdfn}{Definition}

\theoremstyle{remark}

\numberwithin{thm}{section}

\newtheorem*{rem*}{Remark}

\newtheorem*{rems*}{Remarks}
\newtheorem*{ex*}{Example}

\theoremstyle{plain}
\newcounter{zaehler}

\newtheorem{introthm}[zaehler]{Theorem}

\title{On the \textit{K}-theory of pushouts}

\author{Markus Land}
\address{Institut f\"ur Mathematik, Fachbereich 08, Johannes Gutenberg-Universit\"at Mainz, D-55099 Mainz, Germany}
\email{mland@uni-mainz.de}
\author{Georg Tamme}
\email{georg.tamme@uni-mainz.de}

\thanks{The authors were partially supported by the Deutsche Forschungsgemeinschaft (DFG) through CRC 1085. ML was supported by the Danish National Research Foundation through the Copenhagen Centre
for Geometry and Topology (DNRF151). Results incorporated in this paper have
received funding from the European Unions Horizon 2020 research and innovation programme under the
Marie Sklodowska-Curie grant agreement No 888676. GT acknowledges funding by the Deutsche Forschungsgemeinschaft (DFG,
German Research Foundation) through the Collaborative Research Centre TRR 326 ‘Geometry and Arithmetic of Uniformized Structures’, project number 444845124.}

\date{\today}

\begin{document}

\begin{abstract}
We reveal a relation between the behaviour of localizing invariants $E$ on pushouts and on pullbacks of ring spectra. 
More concretely, we show that the failure of $E$ sending a pushout of ring spectra to a pushout is controlled by the value of $E$ on a pullback of ring spectra. Vice versa, in many situations, we show that the failure $E$ of sending a pullback square to a pullback is controlled by the value of $E$ on a pushout of ring spectra. 

The latter can be interpreted as identifying the $\odot$-ring, introduced in earlier work of ours, as a pushout which turns out to be explicitly computable in many cases. This opens up new possibilities for direct computations.
As further applications, we give new proofs of (generalizations) of Waldhausen's seminal results on the $K$-theory of generalized free products and obtain a general relation between the value of a localizing invariant on trivial square zero extensions and on tensor algebras.
\end{abstract}

\maketitle

\setcounter{tocdepth}{2}
\tableofcontents

\section{Introduction}

The purpose of this paper is to investigate the behaviour of localizing invariants on pushouts of ring spectra and to reveal a surprising relation to their behaviour on pullbacks. A prime example of a localizing invariant is given by algebraic $K$-theory and a prime example of a pushout of rings is given by the group ring of an amalgamated product of discrete groups. First results in this situation have been obtained by 
Gersten \cite{Gersten} and Waldhausen \cite{Waldhausen1,Waldhausen2}. One of Waldhausen's motivation to study the algebraic $K$-theory of group rings of amalgamated products was to prove that the Whitehead group of the fundamental group of certain 3-manifolds vanishes. This vanishing result is also implied by the Farrell--Jones conjectures, which have been initiated in seminal work of Farrell and Jones \cite{FJ1} and vastly extended in \cite{Lueck-hyperbolic, Lueck-CAT0,Lueck-lattices, Lueck-arithmetic}. The conjecture states that the algebraic $K$-theory of a group ring is described in terms of generalized equivariant cohomology, see \cite{Lueck-survey} for a survey on this (and related) conjecture and applications.
The algebraic $K$-theory (and related invariants such as Hochschild homology) of group rings is of general interest in algebraic and geometric topology as well as in representation theory and has been widely studied, for instance in the context of the Bass trace conjecture.
This paper gives a new proof of Waldhausen's findings on the $K$-theory of group rings of amalgamated products of groups and extends them to more general pushouts, in particular to possibly non-discrete rings, and to arbitrary localizing invariants. 

As indicated earlier, our results on pushouts of rings are surprisingly closely related
to the behaviour of localizing invariants on pullbacks. The precise nature of this relation rests upon our earlier work \cite{LT}. To explain it, we first summarize the necessary background from op.\ cit.
The main aim therein was to understand under what circumstances a given localizing invariant satisfies \emph{excision}, i.e.\ sends a Milnor square (particular pullback squares consisting of discrete rings) to a pullback square, and to give qualitative results on the failure of excision. This was achieved in the following sense: To a pullback square 
\begin{equation}\label{pullback-square}
\begin{tikzcd}\tag{$\square$}
	A \ar[r] \ar[d] & B \ar[d] \\ 
	A' \ar[r] & B'
\end{tikzcd}
\end{equation}
of ring spectra we have associated a new ring spectrum, the $\wtimes{}{}$-ring, together with a map $\odot \to B'$
which determines the failure of excision on \emph{any} localizing invariant. Unfortunately, we have not described the ring spectrum $\wtimes{}{}$ in a manner which made it appealing for a priori calculations --- it is defined as an endomorphism ring in a suitably defined stable $\infty$-category, and we have only given a single non-trivial calculation of this ring. Fortunately, for many applications of this result, we merely had to make use of very formal properties of this new ring spectrum and its map to $B'$. Much to our surprise, the main results of this paper show that in many cases of interest, the $\wtimes{}{}$-ring is described as an explicitly computable pushout. This fact opens up new possibilities for direct computations and gives an intriguing relation between the a priori unrelated works on pushouts and excision existing in the literature.

To state our main results, we fix a base $\E_2$-ring spectrum $k$ and define a $k$-localizing invariant to be a functor from the $\infty$-category $\Cat_\infty^k$ of small, idempotent complete, $k$-linear $\infty$-categories to a stable $\infty$-category sending exact sequences to fibre sequences.
For a $k$-localizing invariant $E$ and a $k$-algebra $R$, we set $E(R) = E(\Perf(R))$ and we call a commutative square of $k$-algebras a \emph{motivic pullback square} if it is sent to a pullback square by any $k$-localizing invariant. We consider a commutative square of $k$-algebras as follows.
\begin{equation}\label{diag}\tag{$\diamond$}
\begin{tikzcd}
	R \ar[r] \ar[d] & S \ar[d] \\
	R' \ar[r] & S'
\end{tikzcd}
\end{equation}

\begin{introthm}\label{ThmA}
Assume that \eqref{diag} is a pushout square. There is an associated natural $k$-algebra $\widetilde{R}$ and a refined commutative diagram of $k$-algebras
\[\begin{tikzcd}
	R \ar[drr, bend left] \ar[ddr, bend right] \ar[dr] & & \\
		& \widetilde{R} \ar[r] \ar[d] & S \ar[d] \\
		& R' \ar[r] & S' 
\end{tikzcd}\]
whose inner square is a motivic pullback square. The underlying $k$-module of $\widetilde{R}$ is equivalent to the pullback $R' \times_{S \otimes_R R'} S$ and the underlying diagram of $k$-modules is the canonical one.
\end{introthm}

A consequence of the second part of Theorem~\ref{ThmA} is that the fibre of the ring map $R \to \widetilde{R}$ is given by $J\otimes_R I$ where $I = \fib(R \to R')$ and $J = \fib(R \to S)$. We may then think of $J \otimes_R I \to R$ as an ideal, and of $\widetilde{R}$ as the quotient of $R$ by this ideal. With $\widetilde{R}$ replaced by $R/(J \otimes_R I)$, Jeff Smith announced that the resulting square appearing in Theorem~\ref{ThmA} is sent to a pullback by $K$-theory in a talk at Mittag-Leffler in 2006.
A preprint, however, has not appeared.

For computational purposes, one is naturally led to analyze the effect of the map $R \to \widetilde{R}$ on a given localizing invariant. This is, in spirit, similar to the situation in \cite{LT}, where we analyzed the effect of $\wtimes{}{} \to B'$
on localizing invariants. 
For instance, for a Milnor square, $\wtimes{}{}$ is connective and the map $\wtimes{}{} \to B'$ is equivalent to the canonical truncation map $\wtimes{}{} \to \tau_{\leq 0}(\wtimes{}{})$; see \cite[Example 2.9]{LT}. 
As a consequence, for many localizing invariants of interest, one has very good control over the effect of the map $\wtimes{}{} \to B'$. This observation formed the basis of almost all our applications in \cite{LT}.
For instance, on the fibre of the cyclotomic trace or periodic cyclic homology in characteristic 0, this map induces an equivalence.

One may wonder whether similar results hold true for the map $R \to \widetilde{R}$ in place of $\wtimes{}{} \to B'$ in the following situation, somewhat dual to the case of Milnor squares: Namely, the case of a pushout diagram \eqref{diag} in which all rings are discrete, $\widetilde{R}$ is coconnective\footnote{In the case of pullback diagrams where all rings are connective, it follows that $\wtimes{}{}$ is also connective. Dually, however, given a pushout diagram where all rings are coconnective, $\widetilde{(-)}$ need not be coconnective.}, and the map $R \to \widetilde{R}$ identifies with the canonical map $\tau_{\geq0}(\widetilde{R}) \to \widetilde{R}$. This is for instance the case in the situation of Waldhausen's papers \cite{Waldhausen1,Waldhausen2}. Inspired by his results and further basic examples, we raised the question whether in this case, the map $R \to \widetilde{R}$ induces an equivalence in (connective) $K$-theory provided $R$ is regular coherent. We shared this question with Burklund and Levy who soon after were able to show that this is indeed correct under an additional flatness assumption on the negative homotopy groups of $\widetilde{R}$ (which are automatically satisfied in Waldhausen's situation), see \cite{BL}. Their argument is based on the theorem of the heart for $t$-structures of Barwick \cite{Barwick} and Quillen's classical d\'evissage theorem in $K$-theory of abelian categories \cite{Quillen-Higher-K}. As a result, in $K$-theory and under regularity and flatness assumptions, the map $R \to \widetilde{R}$ induces an equivalence. In other localizing invariants like $\THH$ or $\TC$, this is far from true (and closely related to the fact that these invariants are far from being $\A^{1}$-invariant), but one can make the error term somewhat explicit. 

\medskip

Theorem~\ref{ThmA} will be proven as an application of a generalization the main theorem of our earlier work \cite{LT} (Theorem~\ref{thm:GLT} below) and Theorem~\ref{ThmB} below. To explain those, we introduce the following bit of terminology inspired by Milnor's role in excision in $K$-theory and his introduction of what are now called Milnor squares. The notation is chosen to be compatible with that of \cite{LT}.

\begin{introdfn}\label{dfn:Milnor-context}
A \emph{Milnor context} (over $k$) consists of a triple $(A',B,M)$ where $A'$ and $B$ are $k$-algebras and $M$ is a pointed $(B,A')$-bimodule in $\Mod(k)$.
The basepoint of $M$ determines canonical maps $A' \to M \leftarrow B$ of which the first is right $A'$-linear and the second is left $B$-linear.
\end{introdfn}

A pullback square of ring spectra \eqref{pullback-square} as considered in \cite{LT} gives
 rise to the Milnor context $(A',B,B')$ in the evident way.
Classical Milnor squares are hence special cases of Milnor contexts in the above sense. The method of proof of \cite[Main Theorem]{LT} generalizes to this setting and yields the following, see Section~{\ref{sec:bimodule-construction}} for the details.

\begin{thm}[Generalized \cite{LT}]\label{thm:GLT}
To a Milnor context $(A',B,M)$ there are naturally associated $k$-algebras $A$ and $A' \wtimes{A}{M} B$, and a motivic pullback diagram of $k$-algebras 
\begin{equation} \begin{tikzcd}\label{square}
	A \ar[r] \ar[d] & B \ar[d] \\ 
	A' \ar[r] & A' \wtimes{A}{M} B .
\end{tikzcd}
\end{equation}
The underlying $(A,A)$-bimodule of $A$ canonically identifies with the pullback $A' \times_M B$, the underlying $(A',B)$-bimodule of $A' \wtimes{A}{M} B$ canonically identifies with $A' \otimes_A B$, and the underlying diagram of $(A,A)$-bimodules is the canonical one. 
\end{thm}
\begin{Notation}
To emphasize the dependence of the $k$-algebra $A$ in the above theorem on the Milnor context $(A',B,M)$, we will sometimes use the notation $A' \boxtimes_M B$ for the $k$-algebra $A$. 
\end{Notation}

Viewing a pullback diagram \eqref{pullback-square} as a Milnor context as above, the square constructed in Theorem~\ref{thm:GLT} is the one already obtained in \cite[Main theorem]{LT}. In particular, the $k$-algebra $A' \boxtimes_{B'} B$ is canonically equivalent to $A$ (hence the notation). Moreover, in Proposition~\ref{prop:describing-multiplication}, we give a formula for the multiplication map of $A' \wtimes{A}{M} B$.

\medskip

We can now come to the main result of this paper. To state it, let $(A',B,M)$ be a Milnor context, and let $A_0$ be a $k$-algebra equipped with $k$-algebra maps $A' \leftarrow A_0 \to B$ participating in a commutative diagram
\[ \begin{tikzcd}
	A_0 \ar[r] \ar[d] & B \ar[d] \\
	A' \ar[r] & M
\end{tikzcd}\]
of $(A_0,A_0)$-bimodules, hence inducing a canonical map $\phi \colon B\otimes_{A_{0}} A' \to M$ of $(B,A')$-bimodules. Such a datum is equivalently described by a morphism of Milnor contexts (see Definition~\ref{def:map-of-milnor-contexts}) $(A_0,A_0,A_0) \to (A',B,M)$ and we show in Lemma~\ref{lemma:map-of-milnor-contexts} that the canonical $(A_{0},A_{0})$-bimodule map $A_{0} \to A$ refines canonically to a map of $k$-algebras. If the map $\phi$ is an equivalence, we call the morphism $(A_{0},A_{0},A_{0}) \to (A',B,M)$, or sometimes by abuse of notation $A_0$ itself, a \emph{tensorizer} for $(A',B,M)$.

\begin{introthm}\label{ThmB}
Let $(A',B,M)$ be a Milnor context with tensorizer $A_{0}$. Then, the canonical commutative square 
\[
\begin{tikzcd}
	A_0 \ar[r] \ar[d] & B \ar[d] \\
	A' \ar[r] & A' \wtimes{A}{M} B
\end{tikzcd}
\]
obtained from Theorem~\ref{thm:GLT} and the map $A_{0} \to A$ is a pushout diagram in $\Alg(k)$.
\end{introthm}

Thus, given a tensorizer $A_0$ for a Milnor context $(A',B,M)$, the $k$-algebra $A'\wtimes{A}{M}B$ is the pushout $A' \amalg_{A_0} B$.

Theorem~\ref{ThmA} follows from Theorem~\ref{thm:GLT} and Theorem~\ref{ThmB} by considering the Milnor context $(R',S,S \otimes_R R')$ for which $R$ is a tensorizer by construction. 
In addition, it turns out that  there exist tensorizers for many classical examples of Milnor squares, so one can apply Theorem~\ref{ThmB} and consequently identify the $\wtimes{}{}$-ring of \cite{LT}. Some examples are discussed below and in Section~\ref{sec:examples}.

\subsection*{Applications}
In the following, we list some sample applications of the above theorems.
To begin, we go back to Waldhausen's results and first describe his setup in more detail. Following Waldhausen, a map of $k$-algebras $C \to A$ is called a pure embedding if it admits a $(C,C)$-linear retraction. Typical examples are maps of group rings induced by subgroup inclusions. For a span $A \leftarrow C \to B$ of pure embeddings of $k$-algebras, we denote 
by $\bar A$ and $\bar B$ the cofibres of the maps from $C$ to $A$ and $B$, respectively. To obtain information about  the value of a localizing invariant $E$ on the pushout $A \amalg_{C} B$, we ought to analyze the ring $\widetilde C$ given by Theorem~\ref{ThmA}. We show in Proposition~\ref{prop:waldhausen-ring-is-sz} that this ring is a trivial square zero extension $C\oplus \Omega M$ of $C$ by the $(C,C)$-bimodule $\Omega M$, where $M = \bar B \otimes_C \bar A$. Such trivial square zero extensions are closely related to the category of $M$-twisted nilpotent endomorphisms, see Proposition~\ref{prop:twisted-nilpotent}, and consequently $E(C\oplus \Omega M)$ decomposes as a direct sum of $E(C)$ and a Nil-term which we denote by $\Nil E(C;M)$.

\begin{introthm}
Let $E$ be a $k$-localizing invariant and let $A \leftarrow C \to B$ be a span of pure embeddings of $k$-algebras.
Then there is a canonical decomposition
\[
E(A \amalg_{C} B) \simeq \big[ E(A) \oplus_{E(C)} E(B) \big] \oplus \Sigma\Nil E(C;M).
\]
If $A$, $B$, and $C$ are discrete, and $\bar A$ and $\bar B$ are both left $C$-flat, then $A \amalg_{C} B$ is discrete as well, and equals the pushout in the category of discrete rings. If moreover
\begin{enumerate}
\item[-] $E$ is $\A^1$-invariant, or
\item[-] $E$ is $K$-theory and $C$ is right regular Noetherian,
\end{enumerate}
then $\Sigma \Nil E(C;M)$ vanishes.
\end{introthm}

We remark that the final part of this theorem in case of $K$-theory is an application of the result of Burklund and Levy about the $K$-theory of coconnective ring spectra mentioned above \cite{BL} and in fact generalizes, for connective $K$-theory, to regular coherent rings. Examples of $\A^{1}$-invariant localizing invariants include Weibel's homotopy $K$-theory as well as periodic cyclic homology over $\Q$-algebras. 

The above theorem shows that trivial square zero extensions appear naturally in the context of Waldhausen's work. In this situation, the bimodule sits in homological degree $-1$. Classical examples of trivial square zero extensions where the bimodule sits in homological degree $0$ are given by the ring of dual numbers, whose $K$-theory has been studied extensively in the literature \cite{MR291158,Soule,MR952571,BACH,HM}. Motivated by these cases, we are interested in descriptions of the value of localizing invariants on general trivial square zero extensions. In this direction we can offer the following result, which can be viewed as a vast generalization of the relation between $NK$-theory (the failure of $\A^{1}$-invariance of $K$-theory) and the $K$-theory of nilpotent endomorphisms, see e.g.\ \cite[Theorem 7.4 (a)]{Bass} and \cite[\S 3]{Grayson}.

\begin{introthm}\label{introthm:free-vs-sz}
Let $C$ be a $k$-algebra and $M$ a $(C,C)$-bimodule. Then there is a canonical commutative square of $k$-algebras
\[ \begin{tikzcd}
	C \oplus \Omega M \ar[r] \ar[d] & C \ar[d] \\
	C \ar[r] & T_C(M) \, .
\end{tikzcd}
\]
This square is a motivic pullback square.
\end{introthm}
The above theorem is obtained by analyzing the span $C \leftarrow T_C(\Omega M) \to C$, whose structure maps do \emph{not} admit retractions and hence do not appear in Waldhausen's context.
Interestingly, 
the identification of the ring $\widetilde{C}$ in case of Waldhausen's pure embeddings as a trivial square zero extension rests upon a comparison to the case at hand. If $C$ is a right regular Noetherian discrete ring, and $M$ is discrete and left $C$-flat, then as indicated above, it follows from \cite{BL} that the map $K(C\oplus \Omega M) \to K(C)$ is an equivalence. Therefore, also the map $K(C) \to K(T_C(M))$ is an equivalence. This generalizes \cite[Theorems 3 \& 4]{Waldhausen1} which in turn extends earlier work of Gersten \cite[Corollary 3.9]{Gersten}.

As indicated above, the case of dual numbers over $C$, i.e.\ the case where $M=\Sigma C$ (and say $C$ is discrete) is now contained in Theorem~\ref{introthm:free-vs-sz}. Here, we thus obtain an equivalence between the relative $E$-theory of the dual numbers and that of $T_{C}(\Sigma C) = C[t]$ with $|t|=1$ up to a shift.

Finally, we discuss two very classical Milnor squares. First, we consider the two-dimensional coordinate axes over a ring $C$, whose $K$-theory has been studied in characteristic zero in \cite{GRW} and in positive characteristic in \cite{Hesselholt}. The Milnor context $(C[x], C[y], C)$ corresponding to the gluing of two axes along the origin admits $C[x,y]$ as a tensorizer. Using this we obtain the following result.
\begin{introthm}\label{introthmE}
There is a canonical motivic pullback square
\[ \begin{tikzcd}
	C[x,y]/(xy) \ar[r] \ar[d] & C[x] \ar[d] \\
	C[y] \ar[r] & C[t]
\end{tikzcd}\]
where $|t|=2$.
\end{introthm}
Therefore, calculations about the coordinate axes $C[x,y]/(xy)$ are essentially equivalent to calculations about $C[t]$ with $|t|=2$. In Proposition~\ref{eq:K-theory-of-homotopical-polynomial-algebras}, we give an independent calculation of the relative rational $K$-theory of $C[t]$ (in fact, for arbitrary $|t|>0$), thus providing an alternative proof of the results of Geller--Reid--Weibel. Let us briefly also mention the following observation in case of perfect fields $k$ of positive characteristic. In this case, we have that $\THH(k) \simeq k[t]$ with $|t|=2$ by B\"okstedt periodicity. 
Furthermore, the $K$-theory of $\THH(k)$ was recently calculated by Bay\i nd\i r and Moulinos \cite[Theorem 1.2]{BM}. This calculation, together with \cite{Hesselholt}, reveals that the relative $K$-theory of $k[t]$ agrees with the relative $K$-theory of $k[x,y]/(xy)$ up to a shift. Theorem~\ref{introthmE} is a conceptual explanation for this computational observation.

We end this introduction by discussing the Rim square, which was one of the motivating examples for Milnor \cite{Milnor}. Let $p$ be a prime, $C_{p}$ the cyclic group of order $p$, and $\zeta_{p}$ a primitive $p$-th root of unity. The Rim square
\[
\begin{tikzcd}
 \Z[C_{p}] \ar[d]\ar[r] & \Z[\zeta_{p}] \ar[d] \\ 
 \Z \ar[r] & \F_{p} 
\end{tikzcd}
\]
admits the polynomial ring $\Z[x]$ as a tensorizer, and we obtain that the associated $\odot$-ring is given by $\Z[\zeta_{p}]\sslash (\zeta_{p}-1)$, i.e.\ by freely dividing $\Z[\zeta_{p}]$ by $\zeta_{p}-1$ in $\Z$-algebras. Thus we have:

\begin{introthm}
Let $p$ be a prime number. Then there is a canonical motivic pullback square
\[ \begin{tikzcd}
	\Z[C_{p}] \ar[r] \ar[d] & \Z[\zeta_{p}] \ar[d] \\
	\Z \ar[r] & \Z[\zeta_{p}]\sslash (\zeta_{p}-1) .
\end{tikzcd}
\]
\end{introthm}

Independently of our work, Krause and Nikolaus have constructed a motivic pullback square as in the above theorem with $\Z[\zeta_{p}]\sslash (\zeta_{p}-1)$ replaced by $\tau_{\geq 0}(\Z^{tC_{p}})$. Comparing their and our construction, it turns out that the $\Z$-algebras $\Z[\zeta_{p}]\sslash (\zeta_{p}-1)$ and $\tau_{\geq 0}(\Z^{tC_{p}})$ are equivalent, and that their and our motivic pullback squares agree. Since $\tau_{\geq 0}(\Z^{tC_{p}})$ is in turn equivalent to $\F_{p}[t]$ with $|t|=2$, whose $K$-theory is fully understood as explained above, one can infer knowledge about the $K$-theory of $\Z[C_{p}]$.

\subsection*{Acknowledgements}
We thank Robert Burklund and Ishan Levy for discussions about the $K$-theory of coconnective rings and the relation to our work and in particular the categorical formulation of Theorem~\ref{ThmB}. We thank Achim Krause and Thomas Nikolaus for sharing with us their motivic pullback square associated with the Rim square and further discussions about presentations of the $\odot$-ring. We furthermore thank the referee for their careful reading and helpful comments.

\subsection*{Conventions}
We fix an $\E_{2}$-algebra $k$ in spectra.
For a $k$-algebras $A$, $\RMod(A)$ denotes the category of right $A$-modules in $\Mod(k)$ and $\Perf(A) = \RMod(A)^\omega$ the category of perfect (equivalently compact) $A$-modules. For a second $k$-algebra $B$, we denote by $\BMod(B,A)$ the $\infty$-category of $(B,A')$-bimodules in $\Mod(k)$. We evaluate a localizing invariant $E$ on a $k$-algebra $A$ by setting $E(A) := E(\Perf(A))$. 

We make use of the Lurie tensor product on the $\infty$-category $\Pr^{\L}_{\st}$ of presentable and stable $\infty$-categories in the following way. By assumption $\Mod(k)$ is a (commutative, if $k$ is $\E_\infty$) algebra object in $\Pr^{\L}_{\st}$ and we write $\Pr^{\L}(k)$ for the closed (symmetric, if $k$ is $\E_\infty$) monoidal $\infty$-category of $\Mod(k)$-modules in $\Pr^{\L}_{\st}$, or in other words of $k$-linear presentable $\infty$-categories; see \cite[\S4]{HoyoisScherotzkeSibilla}. The internal hom objects are given by $\Fun^{\L}_{k}(-,-)$, the $\infty$-category of colimit preserving $k$-linear functors. If $k$ is $\E_\infty$, we write $\otimes_{k}$ for the relative tensor product $\otimes_{\Mod(k)}$.

\section{The $\odot$-construction for Milnor contexts}
	\label{sec:2}
As indicated in the introduction, the proof of Theorem~\ref{thm:GLT} follows closely the argument of \cite{LT}. Fundamental to this argument was the use of what are called lax pullbacks in \cite{Tamme:excision, LT}, which we choose to call oriented fibre products in this paper, as it is not a lax pullback in the technical sense (rather a partially lax pullback in the sense of \cite{LNP}). The main observation to extend the arguments of \cite{LT} to the present setup is that the construction of \cite{LT} in fact makes sense for arbitrary Milnor contexts in which the bimodule $M$ itself need not be a ring. 
To explain this in more detail, we first recall some constructions in oriented fibre products which will be used several times. 

\subsection{Complements on oriented fibre products}

Here we collect some facts about oriented fibre products. In the context of excision in algebraic $K$-theory, these were first used in \cite{Tamme:excision}, and we refer to \S1 of op.~cit.~for their basic properties. See also \cite[\S1]{LT}, \cite{BachmannCatMilnor} for discussions, and \cite{CDW} for a systematic treatment. Given a diagram of $\infty$-categories $\cA \overset{p}\to \cC \overset{q}\leftarrow \cB$, the \emph{oriented fibre product} $\cA \laxtimes{p,\cC, q} \cB$  is defined as the pullback of $\infty$-categories
\[
\begin{tikzcd}
 \cA \laxtimes{p,\cC, q} \cB \ar[d]\ar[r] & \Fun(\Delta^{1}, \cC) \ar[d] \\ 
 \cA \times \cB \ar[r, "p\times q"] & \cC\times \cC 
\end{tikzcd}
\]
where the right vertical map is the source-target map. This is also a \emph{partially lax limit} in the technical sense, where the morphism $q$ is marked, see \cite[Part 1.4]{LNP}.
Objects in $\cA \laxtimes{p,\cC,q} \cB$ are thus represented by triples $(X,Y,f)$ where $X \in \cA$, $Y\in \cB$, and $f\colon p(X) \to q(Y)$ is a morphism in $\cC$, and mapping spaces are computed by the pullback diagram
\[
\begin{tikzcd}
 \Map_{\slax}( (X,Y,f) , (X', Y', f') ) \ar[d]\ar[r] & \Map_{\cB}(Y,Y') \ar[d, "f^{*}\circ q"] \\ 
 \Map_{\cA}(X,X') \ar[r, "f'_{*}\circ p"] & \Map_{\cC}(p(X), q(Y') ) .
\end{tikzcd}
\]
If $\cA \overset{p}\to \cC \overset{q}\leftarrow \cB$ is a diagram of stable $\infty$-categories and exact functors, then the oriented fibre product is stable, and the mapping spectra obey the same formula.
For readability, we often drop $p$, $q$, or $\cC$ from the notation when these are clear from the context, but note that the orientation of the cospan is relevant in the construction (in other words, we need to keep track which of the two morphisms $p$ and $q$ is marked to form the partially lax limit).

Oriented fibre products are clearly functorial in functors between diagrams. Moreover, they are functorial in natural transformations in the following sense. 
\begin{lemma}\label{lemma:functoriality}
Given a natural transformation $\epsilon\colon p \to p'$, there is an induced functor
\[
\epsilon^{*}\colon \cA \laxtimes{p',\cC, q} \cB \lto \cA \laxtimes{p,\cC, q} \cB
\]
which on objects is given by 
$(X,Y, p'(X) \xrightarrow{f} q(Y)) \mapsto (X,Y, p(X) \xrightarrow{f\circ \epsilon(X)} q(Y))$. Similarly, a natural transformation $\delta\colon q \to q'$ induces a functor
\[
\delta_{*}\colon \cA \laxtimes{p,\cC, q} \cB \lto \cA \laxtimes{p,\cC, q'} \cB
\]
which on objects is given by $(X,Y, p(X) \xrightarrow{f} q(Y)) \mapsto (X,Y, p(X) \xrightarrow{\delta(Y) \circ f} q'(Y))$.
\end{lemma}
\begin{proof}
This follows from the following alternative description of the oriented fibre product. Lurie's unstraightening theorem for left fibrations combined with the identification \cite[Remark A.1.12]{HMS} of left fibrations over $\cA^{\op}\times \cB$ with bifibrations over $\cA \times \cB$   provide canonical equivalences
\[ 
\Fun(\cA^\op\times \cB,\mathrm{Spc}) \simeq \mathrm{LFib}_{/\cA^{\op} \times \cB} \simeq \mathrm{BiFib}_{/\cA\times \cB}.
\]
The $\infty$-category $\cA \laxtimes{p,\cC,q} \cB$ is then given by taking the total space of the bifibration associated to the functor $\Map_{\cC}(p(-),q(-)) \colon \cA^\op \times \cB \to \mathrm{Spc}$ by the above equivalences, because the bifibration associated to the functor $\Map_\cC(-,-)$ is the source-target map $\Fun(\Delta^1,\cC) \to \cC \times \cC$. Clearly, the natural transformations $\epsilon$ and $\delta$ induce maps in $\Fun(\cA^\op\times \cB,\mathrm{Spc})$, hence also maps between the corresponding bifibrations. On total spaces, these give the desired functors. The effect on objects follows from the explicit description of the effect of the unstraightening equivalence on morphisms.
\end{proof}

\begin{lemma}\label{lemma:oriented-pb-adjunction}
In the above situation, assume that the functor $p \colon \cA \to \cC$ admits a right adjoint $r \colon \cC \to \cA$. Then there is a canonical equivalence 
\[
\cA \laxtimes{p, \cC, q} \cB \overset{\simeq}\lto \cA \laxtimes{\id,\cA,r q} \cB
\]
given on objects by $(X,Y, p(X) \xrightarrow{f} q(Y)) \mapsto (X,Y, X \xrightarrow{\hat f} rq(Y))$, where $\hat f$ is the adjoint morphism of $f$.
Similarly, if $q$ admits a left adjoint $s \colon \cC \to \cB$, then there is a canonical equivalence
\[
\cA \laxtimes{p, \cC, q} \cB \overset{\simeq}\lto \cA \laxtimes{sp, \cB, \id} \cB
\]
given on objects by $(X,Y, p(X) \xrightarrow{f} q(Y)) \mapsto (X,Y, sp(X) \xrightarrow{\hat f} Y)$, where $\hat f$ is the adjoint morphism of $f$.
\end{lemma}
\begin{proof}
We only prove the first statement, the second being entirely analogous.
The canonical commutative diagram of functors 
\[ \begin{tikzcd}
	\cA \ar[r,"p"] \ar[d, equal] & \cC  \ar[r] \ar[d,"r"] & \cB \ar[l,"q"'] \ar[d,equal] \\
	\cA \ar[r,"rp"] & \cA & \cB \ar[l,"rq"']
\end{tikzcd}\]
and the unit of the adjunction $\eta\colon \id \to rp$, by Lemma~\ref{lemma:functoriality}, induce the functors 
\[ \cA \laxtimes{p,\cC,q} \cB \lto \cA \laxtimes{rp,\cA,rq} \cB \lto \cA \laxtimes{\id,\cA,rq} \cB,\]
respectively. The composite has the prescribed effect on objects. It follows immediately that this composite functor is essentially surjective, and using the above pullback formula for mapping spaces one obtains that it is also fully faithful.
\end{proof}

We will also need the following slight extension of \cite[Prop.~13]{Tamme:excision}.

\begin{lemma}\label{lemma:compact-generation}
Let $\cA \overset{p}\to \cC \overset{q}\leftarrow \cB$ be a diagram of stable $\infty$-categories. Assume that $\cA$ and $\cB$ are compactly generated presentable, $\cC$ is cocomplete, $p$ and $q$ preserve all small colimits, and $p$ preserves compact objects. Then $\cA \laxtimes{\cC} \cB$ is a compactly generated stable $\infty$-category, and its subcategory of compact objects coincides with $\cA^{\omega} \laxtimes{\cC} \cB^{\omega}$.
\end{lemma}

\begin{proof}
Stability is proven in \cite[Lemma 8]{Tamme:excision}. 
As $p$ and $q$ preserve colimits, colimits in $\cA \laxtimes{\cC} \cB$ are computed component-wise. As moreover $\cA$, $\cB$, and $\cC$ are cocomplete, so is $\cA \laxtimes{\cC} \cB$ (loc.~cit.).
Now let $(X,Y,f)$ be an object of $\cA^{\omega} \laxtimes{\cC} \cB^{\omega}$. We claim that $(X,Y,f)$ is compact in $\cA \laxtimes{\cC} \cB$. Note that the functor corepresented by $(X,Y,f)$ sits in a pullback square
\[
\begin{tikzcd}
 \Map_{\slax}( (X,Y,f), -)  \ar[d]\ar[r] & \Map_{\cB}(Y, -) \ar[d] \\ 
 \Map_{\cA}(X,-) \ar[r] & \Map_{\cC}(p(X), q(-)) .
\end{tikzcd}
\]
As filtered colimits in spaces distribute over pullbacks, the assumptions now imply that $(X,Y,f)$ is compact.
We conclude that the inclusion $\cA^{\omega} \laxtimes{\cC} \cB^{\omega} \hookrightarrow \cA \laxtimes{\cC} \cB$ extends to a fully faithful functor 
\[
\Ind(\cA^{\omega} \laxtimes{\cC} \cB^{\omega}) \hookrightarrow \cA \laxtimes{\cC} \cB.
\]
As by assumption $\cA$ and $\cB$ are compactly generated, it follows that its essential image contains $\cA$ and $\cB$ viewed as full subcategories via the  canonical inclusions $\cA \hookrightarrow \cA \laxtimes{\cC} \cB$ and $\cB \hookrightarrow \cA \laxtimes{\cC} \cB$. As every object in $\cA \laxtimes{\cC} \cB$ is an extension of an object of $\cA$ by an object of $\cB$ and the essential image is a stable subcategory, it follows that it has to be all of $\cA \laxtimes{\cC} \cB$. Thus the latter is compactly generated. As $\cA^{\omega} \laxtimes{\cC} \cB^{\omega}$ is idempotent complete, it follows that $(\cA \laxtimes{\cC} \cB)^{\omega} = \cA^{\omega} \laxtimes{\cC} \cB^{\omega}$.
\end{proof}

\subsection{The $\odot$-theorem for Milnor contexts}
	\label{sec:bimodule-construction}
In this subsection, we prove Theorem~\ref{thm:GLT}.
We recall that a Milnor context consists of a triple $(A',B,M)$ where $A'$ and $B$ are $k$-algebras, and $M$ is a pointed $(B,A')$-bimodule in $\Mod(k)$. The basepoint of $M$ can be described by the following equivalent data:
\begin{enumerate}
\item A $k$-linear map $k \to M$,
\item a right $A'$-linear map $A' \to M$,
\item a left $B$-linear map $B \to M$, or
\item a $(B,A')$-bimodule map $B\otimes_k A' \to M$.
\end{enumerate}

Moreover, we note that the datum of the bimodule $M$ is equivalent to the datum of a colimit preserving and $k$-linear functor $M\colon \RMod(B) \to \RMod(A')$; more precisely, the functor 
\[\begin{tikzcd}[row sep=tiny]
	\BMod(B,A') \ar[r] & \Fun^{\L,k}(\RMod(B),\RMod(A')) \\ M \ar[r,mapsto] & - \otimes_B M 
\end{tikzcd}\]
is an equivalence of categories, where the superscript on the right hand side refers to $k$-linear and colimit preserving functors \cite[Remark 4.8.4.9]{HA}.
We then consider the oriented fibre product
\[
\RMod(A') \laxtimes{M} \RMod(B) := \RMod(A') \laxtimes{\id,\RMod(A'),M} \RMod(B)
\]
whose objects, as we recall, consist of triples $(X,Y,f)$ with $X \in \RMod(A')$, $Y\in \RMod(B)$, and $f\colon X \to Y\otimes_B M$ a map of $A'$-modules.

\begin{rem}
Consider the Milnor context $(A',B,B')$ induced by ring maps $A' \to B' \leftarrow B$. Then by Lemma~\ref{lemma:oriented-pb-adjunction}, the oriented fibre product $\RMod(A') \laxtimes{B'} \RMod(B)$ is canonically equivalent to the one used in \cite{LT} and denoted by $\RMod(A') \laxtimes{\RMod(B')} \RMod(B)$ there. More precisely, in \cite{LT} we used an oriented fibre product of perfect modules, but also these versions are canonically equivalent by Lemma~\ref{lemma:compact-generation}.
\end{rem}

The construction of the oriented fibre product above does not make use of the basepoint $m$ of the bimodule $M$. We utilize $m$ in the following way.
\begin{construction}\label{dfn:lambda-and-A}
We let $\Lambda = (A',B, A' \stackrel{m}{\to} M)$ be the canonical object of $\RMod(A') \laxtimes{M} \RMod(B)$ associated to the base point of $M$. Moreover, we let $A = A' \boxtimes_M B$ be the $k$-algebra 
\[
A := \End_{\slax}(\Lambda).
\] 
\end{construction}
It follows from Lemma~\ref{lemma:compact-generation} that $\Lambda$ is compact. Hence, by the Schwede--Shipley theorem \cite[Theorem 7.1.2.1]{HA} we obtain a fully faithful embedding
\[ 
i \colon \RMod(A) \lto \RMod(A') \laxtimes{M} \RMod(B)
\]
sending $A$ to $\Lambda$. The projections from the oriented fibre product to $\RMod(A')$ and $\RMod(B)$ send $\Lambda$ to $A$ and $B$ respectively. There are therefore canonical $k$-algebra maps $A \to A'$, $A\to B$ such that  the composite functors from $\RMod(A)$ to $\RMod(A')$ and $\RMod(B)$ identify with the corresponding base change functors. In particular, $A'$ and $B$ are endowed with $(A,A)$-bimodule structures. By restriction, $M$ also becomes an $(A,A)$-bimodule.

\begin{rem}\label{rem:comparison-LT1}
For a Milnor context $(A',B,B')$ associated to ring maps $A' \to B' \leftarrow B$, it follows from \cite[Lemma~1.7]{LT} that the $k$-algebra $A=A'\boxtimes_M B$ defined above is in fact simply the pullback $A' \times_{B'} B$ (which is canonically a $k$-algebra). 
\end{rem}

In general, the underlying $(A,A)$-bimodule of $A$ can always be described as a pullback as follows. This proves the second assertion of Theorem~\ref{thm:GLT}.

\begin{lemma}\label{lemma:pullback-bimodule-ring}
There is a canonical pullback diagram
\[\begin{tikzcd}
	A \ar[r] \ar[d] & B \ar[d,"m"] \\
	A' \ar[r,"m"] & M
\end{tikzcd}\]
of $(A,A)$-bimodules.
\end{lemma}
\begin{proof}
By the formula for mapping spectra in oriented fibre products, the following diagram is a pullback. 
\[\begin{tikzcd}
	\map_{\laxtimes{}}(\Lambda,\Lambda) \ar[r] \ar[d] & \map_B(B,B) \ar[d] \\
	\map_{A'}(A',A') \ar[r,"m_*"] & \map_{A'}(A',M)
\end{tikzcd}\]
Here the right vertical map is given by the composite
\[ 
\map_B(B,B) \overset{(-)\otimes_{B}M}\lto \map_{A'}(M,M) \stackrel{m^*}{\lto} \map_{A'}(A',M).
\] 
The above diagram is a commutative diagram of $(A,A)$-bimodules, and it identifies canonically with the diagram in the statement.
\end{proof}

We now construct the $k$-algebra $A' \wtimes{A}{M} B$, the commutative diagram \eqref{square}, and prove that it is a motivic pullback square.

\begin{construction}\label{constr:Q}
We let $\cQ$ be the Verdier quotient of $\RMod(A')\laxtimes{M} \RMod(B)$ by the full subcategory $\RMod(A)$. Denoting by $p$ the projection functor to $\cQ$,
we let $A' \wtimes{A}{M} B$ be the $k$-algebra 
\[
A' \wtimes{A}{M} B := \End_\cQ(p(0,B,0)).
\]
Note that $p(0,B,0)$ is a compact object in $\cQ$.
\end{construction}

\begin{prop}\label{prop:SS}
As a cocomplete $k$-linear $\infty$-category, $\cQ$ is generated by $p(0,B,0)$. In particular, by the Schwede--Shipley theorem, $\cQ$ is equivalent to $\RMod(A'\wtimes{A}{M} B)$.
\end{prop}
\begin{proof}
The oriented fibre product is generated by the two objects $(A',0,0)$ and $(0,B,0)$. The fibre sequence
\[
(0,B,0) \lto (A', B, m) \lto (A',0,0)
\]
shows that  $p(0,B,0) \simeq \Omega p(A',0,0)$. This implies the proposition.
\end{proof}

We write $i_1 \colon \RMod(A) \to \RMod(A')$ and $i_2 \colon \RMod(A) \to \RMod(B)$ for the respective extension of scalars functors and 
\begin{align*}
&j_{1}\colon \RMod(A') \to \RMod(A') \laxtimes{M} \RMod(B), \\
&j_{2}\colon \RMod(B) \to \RMod(A') \laxtimes{M} \RMod(B)
\end{align*}
for the respective inclusion functors. There is then a natural transformation $\tau\colon \Omega j_{1} i_{1} \Rightarrow j_{2} i_{2}$ of functors induced by the canonical fibre sequence
\[ 
(0,Y,0) \lto (X,Y,f) \lto (X,0,0) 
\]
in $\RMod(A')\laxtimes{M} \RMod(B)$.
Consider the following diagram. 
\[
\begin{tikzcd}
	\RMod(A) \ar[r, "i_{2}"] \ar[d, "i_{1}"'] & \RMod(B) \ar[d,"j_{2}"] \ar[ddr, bend left, "k_{2}"] & \\
	\RMod(A') \ar[r,"\Omega j_{1}"] \ar[drr, bend right=15, "k_{1}"'] \arrow[ur, Rightarrow, shorten >= 2em, shorten <= 1em,"\tau"] & \RMod(A') \laxtimes{M} \RMod(B) \ar[dr,"p"] & \\
	&& \cQ
\end{tikzcd}
\]
The functors $k_1$ and $k_2$ are defined so that the outer triangles commute. The transformation $p(\tau)$ is an equivalence, so the outer part of the above diagram commutes. 
Since $k_{1}i_{1}(A) \simeq p(0,B,0) = k_{2}i_{2}(A)$ via $p(\tau)$, the functors in the outer diagram respect the preferred generators. Passing to endomorphism rings, we obtain the commutative diagram \eqref{square} of Theorem~\ref{thm:GLT}.
Arguing exactly as in \cite[p.~888]{LT}, it is a motivic pullback square.
\medskip

To finish the proof of Theorem~\ref{thm:GLT}, it now remains to identify the $(A',B)$-bimodule underlying the $k$-algebra $A' \wtimes{A}{M} B$. 

\begin{prop}\label{prop:underlying-bimodule}
Let $(A',B,M)$ be a Milnor context. Then the $k$-algebra map $A \to A' \wtimes{A}{M} B$ extends uniquely to an $(A',B)$-bimodule map $A' \otimes_A B \to A' \wtimes{A}{M} B$. This map is an equivalence. Under this equivalence, the ring maps $A' \to A' \wtimes{A}{M} B$ and $B \to A' \wtimes{A}{M} B$ correspond to the canonical maps $A' \to A' \otimes_{A} B$ and $B \to A' \otimes_{A} B$, respectively.
\end{prop}
\begin{proof}
First, we note that $A' \wtimes{A}{M} B$ is, by restriction, canonically an $(A',B)$-bimodule, and hence again by restriction an $(A,A)$-bimodule and the map $A \to A'\wtimes{A}{M} B$ is then an $(A,A)$-bimodule map. As $A' \otimes_A B$ is the free $(A',B)$-bimodule on the $(A,A)$-bimodule $A$, there is a unique $(A',B)$-linear extension $A' \otimes_A B \to A' \wtimes{A}{M} B$ as claimed. The last statement in the proposition is clear from this construction. It now remains to see that the map $A' \otimes_A B \to A' \wtimes{A}{M} B$ is an equivalence. 
For this, let 
\[
I = \fib(B \overset{m}{\lto} M),
\]
so $I$ is canonically a $(B, A)$-bimodule. 
As in \cite[Prop.~1.13]{LT}\footnote{We note here that the proofs of Propositions~1.13 and 1.14 in op.~cit.~have never used that $B'$ is a ring, only the formal properties of the oriented fibre product induced by the $(A',B)$-bimodule $B'$. The arguments therefore indeed carry over to our situation.}, we have a fibre sequence of $(B,B)$-bimodules
\begin{equation}\label{ee1}
I \otimes_{A} B \lto B \lto A' \wtimes{A}{M} B,
\end{equation}
where the second map is the $k$-algebra map $B \to A' \wtimes{A}{M} B$. On the other hand, by Lemma~\ref{lemma:pullback-bimodule-ring}, we have a fibre sequence of underlying $(A,A)$-bimodules
\[
I \lto A \lto A'
\]
and by extension of scalars a fibre sequence of $(A,B)$-bimodules
\[
I \otimes_{A} B \lto B \lto A' \otimes_{A} B.
\]
Comparing with \eqref{ee1}, we obtain a canonical equivalence $A' \otimes_{A} B \xrightarrow{\sim} A' \wtimes{A}{M} B$ as $(A,B)$-bimodules. It is now enough to see that this map is in fact $(A',B)$-linear. As this map is $(A,B)$-linear, this map  is the unique $(A,B)$-linear extenstion of the $(A,A)$-linear composite $A ' \to A' \otimes_{A} B \xrightarrow{\sim} A' \wtimes{A}{M} B$, where the first map is the canonical one. Now \cite[Prop.~1.14]{LT} implies that this composite is the $k$-algebra map $A' \to A' \wtimes{A}{M} B$ constructed above. In particular, it is $(A',A)$-linear. Finally, it is a general fact that extending an $(A',A)$-linear map as an $(A,A)$-linear map to an $(A,B)$-linear map yields an $(A',B)$-linear map. This finishes the proof of the proposition.
\end{proof}

To finish this section, we describe the multiplication map of $A' \wtimes{A}{M} B$ in terms of the equivalence with $A' \otimes_A B$ obtained in Proposition~\ref{prop:underlying-bimodule}. 

\begin{prop}\label{prop:describing-multiplication}
The multiplication map of $A' \wtimes{A}{M} B$ is determined by the composite 
\[ B \otimes_A A' \stackrel{\mu}{\lto} M \stackrel{\phi}{\lto}  A' \wtimes{A}{M} B  
\]
by extending left linearly to $A'$ and right linearly to $B$. Here, $\mu$ is the unique $(B,A')$-linear extension of the canonical map $A \to M$, and $\phi$ is induced by the commutative diagram
\[\begin{tikzcd}
	A \ar[r] \ar[d] & B \ar[d] \\
	A' \ar[r] &  A' \wtimes{A}{M} B  
\end{tikzcd}
\]
using that, by Lemma~\ref{lemma:pullback-bimodule-ring},  $M$ is the pushout of the upper-left part of this square.
\end{prop}
\begin{proof}
Since $A' \wtimes{A}{M} B$ is a ring under $A$, its multiplication is given by a map
\[
(A' \wtimes{A}{M} B) \otimes_{A} (A'\wtimes{A}{M} B) \lto A' \wtimes{A}{M} B.
\]
This map is $(A',B)$-linear and as such, by Proposition~\ref{prop:underlying-bimodule}, uniquely determined by its restriction along the map
\[
B \otimes_{A} A' \lto (A' \wtimes{A}{M} B) \otimes_{A} (A'\wtimes{A}{M} B)
\]
which is the tensor product (over $A$) of the ring maps $B \to A' \wtimes{A}{M} B$ and $A' \to A'\wtimes{A}{M} B$. We denote this restricted map by
\[
\nu \colon B \otimes_{A} A' \lto A' \wtimes{A}{M} B.
\]
By construction, the map $\nu$ is $(B,A')$-linear, hence it is uniquely determined by its restriction along $A \to B \otimes_{A} A'$. This restricted map is just the ring map $A \to A' \wtimes{A}{M} B$.

Recall the definition of $\phi\colon M \to A'\wtimes{A}{M} B$ and $\mu\colon B\otimes_{A}A' \to M$.
We claim that $\nu$ is the composition $\phi\circ\mu$. As $B\otimes_{A}A'$ is the free $(B,A')$-bimodule on the $(A,A)$-bimodule $A$ and the restriction of $\phi\circ\mu$ along $A \to B \otimes_{A} A'$  coincides with the ring map $A \to A' \wtimes{A}{M} B$ by construction, it suffices to prove that $\phi\circ\mu$ is $(B,A')$-linear. For $\mu$ this is true by definition. So it suffices to show that $\phi$ is $(B,A')$-linear. By definition, $\phi$ is the unique $(A,A)$-bilinear map fitting in the commutative diagram
\[
\begin{tikzcd}
 A \ar[d]\ar[r] & B \ar[d] \ar[ddr, bend left]\\ 
 A' \ar[r]\ar[drr, bend right] & M \ar[dr, dashed, "\exists! \phi"] \\	
 		&	& A'\wtimes{A}{M} B .
\end{tikzcd}
\]
It thus suffices to produce a $(B,A')$-linear map $M \to A' \wtimes{A}{M} B$ fitting in such a commutative diagram, as then the underlying $(A,A)$-linear map necessarily agrees with $\phi$.
To this end, we use an argument of the proof of \cite[Prop.~1.14]{LT}. Let $L$ denote the localisation endofunctor of $\RMod(A') \laxtimes{M} \RMod(B)$ corresponding to the Verdier projection $p\colon \RMod(A') \laxtimes{M} \RMod(B) \to \cQ$.
Recall the map 
\[
\tau_{\Lambda}\colon (\Omega A', 0,0) \lto (0,B,0)
\]
which becomes an equivalence upon applying $L$. We have the following commutative diagram of mapping spaces in $\RMod(A')\laxtimes{M} \RMod(B)$.
\[
\begin{tikzcd}
 \map((0,B,0), (0,B,0)) \ar[d, "\tau_{\Lambda}^{*}"']\ar[r] & \map( (0,B,0), L(0,B,0) ) \ar[d, "\simeq"] \\ 
 \map((\Omega A', 0,0), (0,B,0)) \ar[r] & \map( (\Omega A', 0, 0), L(0,B,0) )\\
 \map((\Omega A', 0,0), (\Omega A', 0,0)) \ar[u, "\tau_{\Lambda*}"]\ar[r] & \map ( (\Omega A', 0,0), L(\Omega A', 0,0) ) \ar[u, "\simeq"'] 
\end{tikzcd}
\]
The left column identifies with $B \to M \leftarrow A'$, the spectra in the right column all identify with $A' \wtimes{A}{M} B$. Obviously, the middle horizontal map is $(B,A')$-linear. Moreover, the upper respectively lower horizontal map identifies with the canonical ring map $B \to A'\wtimes{A}{M}B$ respectively $A' \to A' \wtimes{A}{M} B$; see also the proof of \cite[Prop.\ 1.14]{LT}. Thus this finishes the construction of the desired map, and hence the proof of the proposition. 
\end{proof}

\begin{rem}
We thank Hongxing Chen and Changchang Xi for sharing with us their construction of non-commuta\-tive tensor products, whose relation to our work we explain now.
In \cite{MR3896125}, they define and study a non-commutative tensor product associated with a so-called exact context, and in 
the preprint \cite{Chen:2012uy}, they also provide $K$-theory pullbacks for what they call homological exact contexts. 
In our language, an exact context is a Milnor context $(A', B, M)$ over $\Z$ in which all objects are discrete and the map $A' \oplus B \to M$ is surjective. In particular, the pullback $A$ is then also discrete. It is homological if the Tor-groups $\Tor_{*}^{A}(A',B)$ vanish in positive degrees.
Their non-commutative tensor product is then a discrete ring, which they denote by $A' \boxtimes_{A} B$ (we apologize for the notational clash with what we denote $A' \boxtimes_M B$). It is a consequence of the definition of $A' \boxtimes_A B$ and Proposition~\ref{prop:describing-multiplication} that there is a canonical ring isomorphism $\pi_{0}(A' \wtimes{A}{M} B) \cong A' \boxtimes_{A} B$. Furthermore, by Proposition~\ref{prop:underlying-bimodule}, an exact context is homological if and only if $A' \odot_{A}^{M} B$ is discrete and hence equivalent to $A' \boxtimes_{A} B$. In particular, \cite[Theorem 1.3(1)]{Chen:2012uy} is a special case of Theorem~\ref{thm:GLT}.
\end{rem}

\subsection{Further properties of the $\odot$-construction}

In this final subsection, we investigate the naturality of the $\odot$-construction as well as its behaviour with respect to base change and possibly existing commutativity structures.
We begin with its functoriality in the Milnor context.

\begin{dfn}\label{def:map-of-milnor-contexts}
A map of Milnor contexts $(A', B, M) \to (C', D, N)$ over $k$  consists of $k$-algebra maps $A' \to C'$ and $B \to D$, a map of $(B,A')$-bimodules  $M \to N$, where $N$ is viewed as a $(B,A')$-module via restriction, and a commutative diagram
\[
\begin{tikzcd}
M \ar[rr] && N \\
& k \ar[ur]\ar[ul],
\end{tikzcd}
\]
where the  maps from $k$ are the respective base points.
\end{dfn}

\begin{lemma}\label{lemma:map-of-milnor-contexts}
Given a map of Milnor contexts $(A', B, M) \to (C', D, N)$ over $k$, we obtain a canonical map of commutative squares of $k$-algebras
\[
\left(\begin{tikzcd}
 A \ar[d]\ar[r] & B \ar[d] \\ 
 A' \ar[r] & A' \wtimes{A}{M} B 
\end{tikzcd}\right)
\lto
\left(
\begin{tikzcd}
 C \ar[d]\ar[r] & D \ar[d] \\ 
 C' \ar[r] & C' \wtimes{C}{N} D 
\end{tikzcd}
\right)
\]
where $A=A' \boxtimes_{M} B$ and $C=C' \boxtimes_{N} D$.
\end{lemma}

\begin{proof}
We will construct a canonical functor 
\[
\Phi\colon \RMod(A') \laxtimes{M} \RMod(B) \lto \RMod(C') \laxtimes{N} \RMod(D)
\]
which respects the objects of Constructions~\ref{dfn:lambda-and-A} and \ref{constr:Q} and therefore induces a functor $\Psi\colon \cQ_M \to \cQ_N$ of the Verdier quotients which again respects the objects of Construction~\ref{constr:Q}. The functor $\Psi$ then participates in the map of squares of categories 
\[ \left(\begin{tikzcd}
\RMod(A) \ar[r] \ar[d] & \RMod(B) \ar[d] \\
\RMod(A') \ar[r] & \cQ_M 
\end{tikzcd}\right)
\lto 
\left(\begin{tikzcd}
\RMod(C) \ar[r] \ar[d] & \RMod(D) \ar[d] \\
\RMod(C') \ar[r] & \cQ_N 
\end{tikzcd}\right), \]
and the map of squares of $k$-algebras in the statement of the lemma is obtained by passing to endomorphisms of the canonical generators everywhere.

We now proceed to construct the functor $\Phi$ above.
We consider the $(D,A')$-bimodule $D \otimes_{B} M$ and view it as a $(B,A')$-bimodule. As such, it receives a map from $M$ induced by the $k$-algebra map $B \to D$. This map of $(B,A')$-bimodules gives a natural transformation of the corresponding functors $\RMod(B) \to \RMod(A')$. From Lemma~\ref{lemma:functoriality}, we obtain a functor
\[
\RMod(A') \laxtimes{M} \RMod(B) \lto \RMod(A') \laxtimes{D\otimes_{B} M} \RMod(B). 
\]
On the other hand, we consider the free $(D,C')$-bimodule on $M$, namely $D \otimes_{B}M\otimes_{A'}C'$. It induces the lower horizontal functor in the following commutative diagram.
\[
\begin{tikzcd}[column sep =   2.5cm]
 \RMod(B) \ar[d, "-\otimes_{B}D"']\ar[r, "-\otimes_{B}D\otimes_{B}M"] & \RMod(A') \ar[d, "-\otimes_{A'}C'"] \\ 
 \RMod(D) \ar[r, "-\otimes_{B}D \otimes_{B}M\otimes_{A'}C'"] & \RMod(C') 
\end{tikzcd}
\]
Here the vertical functors are given by extension of scalars. By functoriality of the pullback of $\infty$-categories and of arrow categories, this commutative diagram induces a functor
\[
\RMod(A') \laxtimes{D\otimes_{B} M} \RMod(B) \lto \RMod(C') \laxtimes{D \otimes_{B}M\otimes_{A'}C'} \RMod(D).
\]
Now the $(B,A')$-bimodule map $M \to N$ is equivalently given by a $(D,C')$-bimodule map $D \otimes_{B}M\otimes_{A'}C' \to N$. By the same argument as in the first step, we obtain a third functor
\[
\RMod(C') \laxtimes{D \otimes_{B}M\otimes_{A'}C'} \RMod(D) \lto \RMod(C') \laxtimes{N} \RMod(D).
\]
Composing the above three functors, we obtain the desired functor $\Phi$. It sends an object $(X,Y, X \xrightarrow{f} Y \otimes_{B} M)$ to 
\[
(X \otimes_{A'} C', Y \otimes_{B} D, X \otimes_{A'} C' \xrightarrow{\hat f} (Y \otimes_{B} D) \otimes_{D} N \simeq Y \otimes_{B} N)
\]
where $\hat f$ is the $C'$-linear extension of the composition $X \xrightarrow{f} Y \otimes_{B}M \to Y\otimes_{B} N$.
In particular, $\Phi$ sends the object $\Lambda_{M} = (A', B, A' \to M)$ of Construction~\ref{dfn:lambda-and-A} to the corresponding object $\Lambda_{N}$ for the Milnor context $(C',D,N)$. Here we use the compatibility of the base points. 
The functor $\Phi$ has all desired properties.
\end{proof}

\begin{rem}
There is a canonical $\infty$-category of Milnor contexts, which in the notation of \cite{HA} is given by $\BMod(\Mod(k))_{(k,k,k)/}$, see \cite[Example 4.3.1.15]{HA}. Here, $\BMod(\Mod(k))$ is an $\infty$-category of bimodules whose objects consist of triples $(A',B,M)$ where $A'$ and $B$ are $k$-algebras and $M$ is a $(B,A')$-bimodule in $\Mod(k)$ and $(k,k,k)$ is the obvious object in $\BMod(\Mod(k))$.
The construction sending $(A',B,M)$ to the commutative square 
\[ \begin{tikzcd}
	\RMod(A) \ar[r] \ar[d] & \RMod(B) \ar[d] \\
	\RMod(A') \ar[r] & \cQ_M
\end{tikzcd}\]
ought to refine to a functor $\BMod(\Mod(k))_{(k,k,k)/} \to \Fun(\Delta^1 \times \Delta^1, \Pr^\L_\omega(k)_{\ast/})$, but we refrain from attempting a precise proof of this. Lemma~\ref{lemma:map-of-milnor-contexts} records what the effect of this supposed functor is on morphisms.
\end{rem}

Next, we record two compatibilities of the $\odot$-construction with tensor products and base change, respectively. 
In the first case, we assume that the base ring $k$ is an $\E_{\infty}$-ring so that $\Mod(k)$ and also $\Alg(k)$ become symmetric monoidal $\infty$-categories with respect to the relative tensor product $\otimes_{k}$.
As always, let $(A', B, M)$ be a Milnor context over $k$. Let $R$ be another $k$-algebra. Then we can form the $k$-algebras $A'_{R} := A' \otimes_{k} R$ and $B_{R} := B \otimes_{k} R$, and the $(B_{R}, A'_{R})$-bimodule $M_{R} := M \otimes_{k} R$. So $(A'_{R}, B_{R}, M_{R})$ is another Milnor context over $k$. 

\begin{prop}\label{prop:tensor-product}
In the above situation, there are canonical equivalences of $k$-algebras 
\[
(A'\boxtimes_{M}B) \otimes_{k} R \xrightarrow{\sim} A'_{R} \boxtimes_{M_{R}} B_{R} \qquad\text{ and }\qquad (A' \wtimes{A}{M} B) \otimes_{k} R \xrightarrow{\sim} A'_{R} \wtimes{A_{R}}{M_{R}} B_{R}.
\]
\end{prop}
\begin{proof}
By assumption $\Mod(k)$ is a commutative algebra object in $\Pr^{\L}_{\st}$, and we write $\Pr^{\L}(k)$ for the closed symmetric monoidal $\infty$-category of $\Mod(k)$-modules in $\Pr^{\L}_{\st}$.
The association $S \mapsto \RMod(S)$ refines to a symmetric monoidal functor $\Alg(k) \to \Pr^{\L}(k)$; see \cite[Remark 4.8.5.17]{HA}. In particular, for any $k$-algebra $S$, we have a canonical equivalence $\RMod(S \otimes_{k} R) \cong \RMod(S) \otimes_{k} \RMod(R)$. 
We claim that all categorical constructions performed in the proof of Theorem~\ref{thm:GLT} are compatible with tensoring with $\RMod(R)$. We first argue that
\begin{equation}
	\label{ee2}
 (\RMod(A') \laxtimes{M} \RMod(B) ) \otimes_{k} \RMod(R) \simeq \RMod(A'_{R}) \laxtimes{M_{R}} \RMod(B_{R}).
\end{equation}
For ease of notation, we write $\cR = \RMod(R)$. By definition of the oriented fibre product, it suffices to check that $(-)\otimes_{k}\cR$ preserves pullbacks of $k$-linear $\infty$-categories and the formation of arrow categories. As $\cR$ is compactly generated, it is dualizable in $\Pr^{\L}(k)$ \cite[Proposition~4.10]{HoyoisScherotzkeSibilla}. Let $\cR^{\vee}$ denote its dual. By the above description of the internal hom objects, we have an equivalence of functors 
$
(-) \otimes_{k} \cR \simeq \Fun^{L}_{k}(\cR^{\vee}, -).
$
Now for any $\cC \in \Pr^{\L}(k)$ we get
\[
\Fun(\Delta^{1}, \cC) \otimes_{k} \cR
	\simeq \Fun^{\L}_{k}(\cR^{\vee}, \Fun(\Delta^{1}, \cC))
	\simeq \Fun(\Delta^{1}, \Fun^{\L}_{k}(\cR^{\vee}, \cC))
	\simeq \Fun(\Delta^{1}, \cC \otimes_{k} \cR).
\]
For the equivalence we use the explicit description of the $\infty$-category of $k$-linear functors \cite[Lemma~4.8.4.12]{HA} together with the fact that colimits in $\Fun(\Delta^{1}, \cC)$ are formed pointwise. By a similar argument one sees that $(-) \otimes_{k} \cR$ preserves pullbacks of $k$-linear categories, as claimed. So we have established the equivalence \eqref{ee2}.
Under this equivalence, the object $\Lambda \otimes_{k} R$, where $\Lambda$ is as in Construction~\ref{dfn:lambda-and-A}, is mapped to the object $\Lambda_{R} := (A'_{R}, B_{R}, A'_{R} \to M_{R})$ corresponding to the Milnor context $(A'_{R}, B_{R}, M_{R})$ in the right-hand $\infty$-category. It follows that 
\[
A \otimes_{k} R \simeq \End_{\slax}(\Lambda) \otimes_{k} \End_{R}(R) \simeq \End_{\slax}(\Lambda_{R}) \simeq A_{R},
\]
establishing the first claim. It also follows that the full subcategory $\RMod(A_{R})$ of the right-hand $\infty$-category in \eqref{ee2} is equivalent to $\RMod(A) \otimes_{k} \cR$. As $(-)\otimes_{k} \cR$ preserves exact sequences of $k$-linear categories (see the proof of~\cite[Lemma~3.12]{LT}), it follows that for the corresponding Verdier quotient $\cQ_{R}$ from Construction~\ref{constr:Q} we also get an equivalence $\cQ \otimes_{k} \cR \simeq \cQ_{R}$, which again preserves the preferred generators. By the construction in \ref{constr:Q}, this implies the second assertion of the proposition.
\end{proof}

One can use very similar arguments to prove the following base change compatibility of the $\odot$-construction. Back to the general setting, we here assume that the base ring $k$ is $\E_{2}$ and consider a map of $\E_{2}$-rings $k \to \ell$. Then $(-)\otimes_{k} \ell$ refines to a monoidal functor $\Mod(k) \to \Mod(\ell)$ and hence induces a functor $\Alg(k) \to \Alg(\ell)$. Both functors will be referred to as  base change along $k \to \ell$. Given a Milnor context $(A', B, M)$ over $k$, we obtain a Milnor context $(A'_{\ell}, B_{\ell}, M_{\ell})$ over $\ell$ by case change.

\begin{prop}\label{prop:base-change}
In the above situation, there are canonical equivalences of $\ell$-algebras
\[
(A'\boxtimes_{M}B) \otimes_{k} \ell \simeq (A'_{\ell}\boxtimes_{M_{\ell}}B_{\ell}) \qquad \text{ and } \qquad (A' \wtimes{A}{M} B) \otimes_{k} \ell \simeq A'_{\ell} \wtimes{A_{\ell}}{M_{\ell}} B_{\ell}. 
\]
\end{prop}

Finally, we remark that our construction is generically incompatible with possibly existing commutativity structures. The main observation for this is the following proposition. Below it will be used to show that diagram \eqref{square} provided by Theorem~\ref{thm:GLT} does not refine to one of $\E_{2}$-$k$-algebras in concrete examples.

\begin{prop}\label{prop:multiplicative-structures}
Let $(A',B,M)$ be a Milnor context. Suppose that in the diagram 
\[ \begin{tikzcd}
	A \ar[r] \ar[d] & B \ar[d] \\
	A' \ar[r] & A' \wtimes{A}{M} B
\end{tikzcd}\]
provided by Theorem~\ref{thm:GLT}, the maps $A' \leftarrow A \to B$ refine to maps of $\E_2$-$k$-algebras. Then $A' \otimes_A B$ is naturally an $\E_1$-$k$-algebra. If furthermore the whole diagram above refines to one of $\E_2$-rings, then this $\E_1$-$k$-algebra $A' \otimes_A B$ is equivalent to $A' \wtimes{A}{M} B$.
\end{prop}
\begin{proof}
Suppose the diagram commutes as a diagram of $\E_2$-$k$-algebras.
The $\E_{2}$-map $A \to A'$ makes $A'$ an $\E_{1}$-$A$-algebra. The extension of scalars functor $\Mod(A) \to \Mod(B)$ is monoidal, and hence refines to a functor $\Alg_{\E_{1}}(\Mod(A)) \to \Alg_{\E_{1}}(\Mod(B))$, which is again left adjoint to the forgetful functor. In particular,  $A' \otimes_{A} B$ is canonically an $\E_{1}$-$B$-algebra, and the canonical map $A' \otimes_{A} B \to A' \wtimes{A}{M} B$ (the right $B$-linear extension of $A' \to A' \wtimes{A}{M} B$) refines to an $\E_{1}$-map. By Proposition~\ref{prop:underlying-bimodule} it is an equivalence.
\end{proof}

\begin{ex}
In (the proof of) Theorem~\ref{thm:free-vs-sz} below, we show that diagram~\eqref{square} associated  to the Milnor context $(\Z,\Z,\Z\oplus \Z)$ identifies with the commutative diagram
\[ 
\begin{tikzcd}
	\Z\oplus \Omega \Z \ar[r] \ar[d] & \Z \ar[d] \\
	\Z \ar[r] & \Z[x].
\end{tikzcd}
\]
Here $\Z \oplus \Omega \Z$ denotes the trivial square zero extension on $\Omega \Z$, and $\Z[x]$ is the usual polynomial algebra.
We note that all maps appearing in this diagram are canonically $\E_\infty$ and even more, there \emph{exists} an $\E_\infty$-homotopy rendering the diagram commutative. However, the $\E_1$-homotopy which is provided by Theorem~\ref{thm:free-vs-sz} does not even refine to an $\E_2$-homotopy. Indeed, if it would, then by Proposition~\ref{prop:multiplicative-structures} we would find that $\Z[x]$ were equivalent to $\Z \otimes_{\Z\oplus \Omega \Z} \Z$ which, however, is isomorphic to $\Gamma_\Z(x)$, the free divided power algebra, not the polynomial algebra.
\end{ex}

\begin{ex}\label{ex:coordinate-axes-multiplication}
Similarly, the $\E_1$-homotopy in the diagram
\[ \begin{tikzcd}
	\F_p[x,y]/(xy) \ar[r] \ar[d] & \F_p[y] \ar[d] \\
	\F_p[x] \ar[r] & \F_p[t]
\end{tikzcd}\]
obtained in Corollary~\ref{cor:coordinate-axes}, where $\F_{p}[t]$ is the free $\F_{p}$-algebra on a generator of degree 2, does not refine to an $\E_2$-homotopy.  Otherwise $\F_p[t]$ would have to have a divided power structure on its positively graded homotopy, which it doesn't as $t^p \neq 0$. Indeed, the relative tensor product $\F_p[x] \otimes_{\F_p[x,y]/(xy)} \F_p[y]$ is canonically an animated commutative ring, and every animated commutative ring has a divided power structure on the ideal of positively graded elements, see e.g.\ \cite[\S 4]{Richter}.
\end{ex}

\section{Proof of Theorems~\ref{ThmA} \& \ref{ThmB}}\label{sec:3}

In this section we aim to prove Theorem~\ref{ThmB} from the introduction, which asserts that a certain square of $k$-algebras is a pushout square. Using that the functor $A \mapsto \Perf(A)$ from $k$-algebras to small $k$-linear categories detects pushouts (see the proof of Theorem~\ref{thm:thm-B-im-Text} for details), it suffices to prove the analogous statement for small $k$-linear $\infty$-categories. We therefore start by describing pushouts in $\Cat^{k}_\infty$ and prove a categorical version (Theorem~\ref{thm:pushout-in-Cat-perf}) of Theorem~\ref{ThmB}. In this form, our result was already used by Burklund and Levy in \cite[\S 4]{BL} see e.g.\ Theorem 4.11 there.

\subsection{Pushouts in $\Cat^{k}_{\infty}$}
	\label{sec:pushouts-in-Cat-perf}

Consider a span 
\[
\cA' \stackrel{u}{\longleftarrow} \cA_0 \stackrel{v}{\lto} \cB
\]
of small $k$-linear $\infty$-categories. Abusing notation, we denote the induced functors on Ind-completions still by $u$ and $v$, and the Yoneda embeddings $\cA' \to \Ind(\cA')$, $\cB \to \Ind(\cB)$ by the identity.
The functor $v$ admits an Ind-right adjoint $v_* \colon \Ind(\cB) \to \Ind(\cA_{0})$ which we may use to form the oriented fibre product 
\[
\cA' \laxtimes{\id, uv_{*}} \cB := \cA' \laxtimes{\id, \Ind(\cA'), uv_{*}} \cB
\]
whose objects consist of triples $(X,Y,X \to uv_*(Y))$, with $X \in \cA'$ and $Y \in \cB$.
We consider the canonical functor 
\begin{equation}
	\label{eq:functor-from-A0}
i_{0}\colon \cA_{0} \lto \cA' \laxtimes{\id, uv_{*}} \cB, \quad Z \mapsto (u(Z), v(Z), u(Z \to v_{*}v(Z)) )
\end{equation}
and let $\cA \subseteq \cA' \laxtimes{\id, uv_{*}} \cB$ denote the thick subcategory generated by the image of this functor. Furthermore, we let $\cQ$ denote the Verdier quotient of $\cA' \laxtimes{\id, uv_*} \cB$ by the full subcategory $\cA$:
\[
\cQ := (\cA' \laxtimes{\id, uv_{*}} \cB) / \cA
\]

\begin{ex}\label{example:span-tensorizer}
If the span $\cA' \leftarrow \cA_0 \to \cB$ is obtained from a span $A' \leftarrow A_0 \to B$ of $k$-algebras by applying $\Perf(-)$, then the oriented fibre product described above is simply given by the compact objects of oriented fibre product associated to the Milnor context $(A',B,B\otimes_{A_0} A)$ as in Section~\ref{sec:bimodule-construction}, and $\cA$ identifies with $\Perf(A)$. Indeed, by construction $\cA$ is given by the thick subcategory generated by $i_0(A_0)$, which is precisely the object $\Lambda$ considered in Construction~\ref{dfn:lambda-and-A}.
\end{ex}

We let $p \colon \cA' \laxtimes{\id, uv_{*}} \cB  \lto \cQ$ be the canonical functor and denote by $j_{1}$ respectively $j_{2}$ the inclusions 
\begin{align*}
& j_{1}\colon \cA' \lto \cA' \laxtimes{\id, uv_{*}} \cB \\
& j_{2}\colon \cB \lto \cA' \laxtimes{\id, uv_{*}} \cB .
\end{align*}
As in the paragraph following Proposition~\ref{prop:SS}, there is a natural transformation $\tau \colon \Omega j_1u \Rightarrow j_2v$ as indicated in the diagram
\begin{equation}
	\label{diag1}
\begin{tikzcd}
	\cA_{0} \ar[r, "v"] \ar[d, "u"'] & \cB \ar[d,"j_{2}"] & \\
	\cA' \ar[r,"\Omega j_{1}"'] \arrow[ur, Rightarrow, shorten >= 1em, shorten <= 1em,"\tau"] & \cA' \laxtimes{\id, uv_{*}} \cB \,
\end{tikzcd}
\end{equation}
induced by the natural fibre sequence
\[
(0,Y,0) \lto (X,Y,f) \lto (X,0,0).
\]
By construction, $p(\tau)$ is a natural equivalence between functors $\cA_{0} \to \cQ$, and we therefore obtain the commutative diagram
\begin{equation}
	\label{diag:pushout}
\begin{tikzcd}
 \cA_{0} \ar[d, "u"]\ar[r, "v"] & \cB \ar[d, "pj_{2}"] \\ 
 \cA' \ar[r, "p\Omega j_{1}"] & \cQ .
\end{tikzcd}
\end{equation}

The promised categorical version of Theorem~\ref{ThmB} is then the following assertion.

\begin{thm}
	\label{thm:pushout-in-Cat-perf}
Diagram \eqref{diag:pushout} is a pushout diagram in $\Cat^k_{\infty}$.  
\end{thm}

For the proof, we need some preliminary lemmas. We denote by $\Pr^{\L}_\omega(k) \subset\Pr^{\L}(k)$ the subcategory of compactly generated, presentable $k$-linear categories with colimit preserving $k$-linear functors that preserve compact objects. 
As $\Perf(k)$ is rigid, taking the Ind-completion induces an equivalence $\Cat^k_{\infty} \simeq \Pr^{\L}_\omega(k)$; see \cite[Prop.~4.5, 4.9]{HoyoisScherotzkeSibilla}. We may therefore equivalently show that the diagram of Ind-completions $\Ind(\eqref{diag:pushout})$ is a pushout in $\Pr^{\L}_\omega(k)$. So from now on we pass to the Ind-completions of the above categories without mentioning that in the notation. We make use of this passage in the following lemma.

\begin{lemma}
	\label{lemma:equiv-of-oriented-pb}
There is a canonical equivalence
\[
\Phi\colon  \cB \laxtimes{v_{*},u_{*}} \cA'   \stackrel{\simeq}{\lto} \cA' \laxtimes{\id,uv_{*}} \cB
\]
fitting in a commutative diagram
\[
\begin{tikzcd}
 \cA' \ar[d, equal] & \cB \laxtimes{v_{*},u_{*}} \cA' \ar[l, "\pr_{\cA'}"'] \ar[r, "\pr_{\cB}"] \ar[d, "\simeq", "\Phi"']& \cB \ar[d, equal]\\
 \cA' & \cA' \laxtimes{\id,uv_{*}} \cB \ar[l, "\Sigma j_{1*}"']  \ar[r, "\pr_{\cB}"] & \cB,
\end{tikzcd}
\]
where $j_{1*}$ denotes a right adjoint of $j_{1}$.
On objects,  $\Phi$ is given by
\[
(Y, X, v_{*}(Y) \xrightarrow{f} u_{*}(X) ) \mapsto ( \fib(\hat f), Y, \fib(\hat f) \xrightarrow{\mathrm{can}} uv_{*}(Y))
\]
where $\hat f\colon uv_{*}(Y) \to X$ is the adjoint map to $f$.
\end{lemma}

\begin{proof}
Lemma~\ref{lemma:oriented-pb-adjunction} provides an equivalence
\[
\cB \laxtimes{v_{*},u_{*}} \cA' \simeq \cB \laxtimes{uv_{*},\id} \cA', \quad 
	(Y, X, v_{*}(Y) \xrightarrow{f} u_{*}(X) ) \mapsto (Y,X, uv_{*}(Y) \xrightarrow{\hat f} X)
\]
which is compatible with the projections $\pr_{\cA'}$ and $\pr_{\cB}$.
As $\cA'$ is stable, there is further a canonical equivalence
\[
\cB \laxtimes{uv_{*},\id}  \cA' \simeq \cA' \laxtimes{\id,uv_{*}}  \cB  
\]
given on objects by
\[
(Y, X,   uv_{*}(Y) \xrightarrow{\hat f} X) \mapsto (\fib(\hat f), Y,  \fib(\hat f) \xrightarrow{\mathrm{can}}  uv_{*}(Y) ) .
\]
By \cite[Lemma 1.5]{LT} the right adjoint $\Sigma j_{1*}$ of $\Omega j_1$ is given by the formula $(X,Y,f) \mapsto \cof(f)$. It follows that the above equivalence fits in a commutative diagram
\[
\begin{tikzcd}
 \cA' \ar[d, equal] & \cB \laxtimes{uv_{*},\id} \cA' \ar[l, "\pr_{\cA'}"']\ar[r, "\pr_{\cB}"] \ar[d, "\simeq"]& \cB \ar[d, equal]\\
 \cA' &  \cA' \laxtimes{\id, uv_{*}} \cB \ar[l, "\Sigma j_{1*}"']  \ar[r, "\pr_{\cB}"] & \cB.
\end{tikzcd}
\]
Composing the two equivalences we obtain the lemma.
\end{proof}

We now consider the Ind-completion of diagram \eqref{diag1} (again without adding this in the notation), pass to the diagram of right-adjoints, and use the equivalence constructed in Lemma~\ref{lemma:equiv-of-oriented-pb} to obtain the following diagram.
\[
\begin{tikzcd}
	\cA_{0}  & \cB \ar[l, "v_{*}"'] \arrow[dl, Rightarrow, shorten >= 1em, shorten <= 1em,"\sigma'"'] \\
	\cA' \ar[u, "u_{*}"]  &  \cB \laxtimes{v_{*},u_{*}}  \cA'  \ar[l, "\pr_{\cA'}"] \ar[u, "\pr_{\cB}"']
\end{tikzcd}
\]

\begin{lemma}\label{claim1}
This diagram identifies with the canonical diagram associated to the oriented fibre product of the diagram $\cA' \overset{u_{*}}{\longrightarrow} \cA_{0} \overset{v_{*}}{\longleftarrow}  \cB$. In particular, restricted to the full subcategory of $\cB \laxtimes{v_*,u_*} \cA'$ on the objects $(Y,X,f \colon v_*(Y) \to u_*(X))$ where $f$ is an equivalence, the transformation $\sigma'$ becomes an equivalence, and the resulting diagram is a pullback in $\widehat{\Cat}_{\infty}$, the $\infty$-category of large $\infty$-categories.

\end{lemma}
\begin{proof}
We first recall that given a pair of adjunctions $(F,G,\epsilon,\eta)$ and $(F',G',\epsilon',\eta')$ between two $\infty$-categories and a natural transformation $\tau\colon F \Rightarrow F'$, the adjoint natural transformation is given by the composite
\[ G' \stackrel{\eta G'}{\Longrightarrow} GFG' \stackrel{G\tau G'}{\Longrightarrow} GF'G' \stackrel{G\epsilon'}{\Longrightarrow} G.\]
Therefore, what we need to calculate is the natural transformation
\[ 
G'\Phi \Longrightarrow GFG'\Phi \Longrightarrow GF'G' \Phi \Longrightarrow G\Phi
\]
where 
\[ 
F = \Omega j_1 u, \quad F' = j_{2}v ,\quad G = u_* \Sigma j_{1*}, \quad \text{ and }  \quad G' = v_* \pr_{\cB}^{2}
\]
and $\Phi$ is the equivalence constructed in Lemma~\ref{lemma:equiv-of-oriented-pb}. 
Here, in order to avoid confusion, we temporarily denote the projection $\cB \laxtimes{v_{*}, u_{*}} \cA' \to \cB$ by $\pr_{\cB}^{1}$ and the projection $\cA' \laxtimes{\id, uv_{*}} \cB \to \cB$ by $\pr_{\cB}^{2}$.
Spelling out all definitions of these functors, we obtain the following canonical equivalences
\[ 
\Sigma j_{1*}j_2 = uv_*, \quad \Sigma j_{1*} \Phi = \pr_{\cA'}, \quad \text{ and } \quad \pr_\cB^{2} \Phi = \pr_{\cB}^{1}
\]
and therefore the map we investigate is given by a composite
\begin{equation} \label{eq:composite}
v_* \pr_\cB^{1} \Longrightarrow u_*uv_* \pr_\cB^{1} \Longrightarrow u_*u(v_*v)v_* \pr_\cB^{1} \Longrightarrow u_* \pr_{\cA'} 
\end{equation}
of natural transformations between functors $\cB \laxtimes{v_*,u_*} \cA' \to \cA_0$,
where the first map is given by the unit of the adjunction $(u,u_*)$ and the second map by the unit of the adjunction $(v,v_*)$. The latter follows by observing that under the equivalence $\Sigma j_{1*} \Omega j_1 u = u$ (recall that $j_1$ is fully faithful), the transformation $\tau$, by definition, induces the unit map
\[ 
u \simeq \Sigma j_{1*} \Omega j_1 u \stackrel{\Sigma j_{1*} \tau u}{\Longrightarrow}  \Sigma j_{1*}j_2 v = uv_*v.
\]

The final map in composite \eqref{eq:composite} arises by applying the functor $u_{*}\Sigma j_{1*}$ to the counit of the adjunction for $(F',G') = (j_{2}v,\pr_\cB^{2} v_*)$, which factors as the counit of $(v,v_*)$ followed by the counit of $(j_2,\pr_\cB^{2})$, on the functor $\Phi$. Therefore, using the above equivalences, the final map in \eqref{eq:composite} is given by the composite
\[ 
u_*uv_*(vv_*)\pr_\cB^{1} \Longrightarrow u_*uv_*\pr_\cB^{1} \stackrel{\alpha}{\Longrightarrow} u_*\pr_{\cA'},
\]
whose first map is induced by the counit of $(v,v_*$).
By the triangle identities, we therefore find that the composite \eqref{eq:composite} we wish to calculate is given by 
\begin{equation}
	\label{eq:second-composite}
v_*\pr_\cB^{1} \Longrightarrow u_*uv_*\pr_{\cB}^{1} \stackrel{\alpha}{\Longrightarrow} u_*\pr_{\cA'}
\end{equation}
where the first map is the unit of $(u,u_*)$. 
Now we describe the natural transformation $\alpha$ which is induced by the counit of $(j_2,\pr_\cB^{2})$ on the functor $\Phi$, or, equivalently, by the counit of the adjunction $(\Phi^{-1}j_{2}, \pr_{\cB}^{1}=\pr^{2}_{\cB} \Phi)$. We first note that on an object $(F, Y, F \to uv_{*}(Y))$ of $\cA' \laxtimes{\id, uv_{*}} \cB$, the counit of the adjunction $(j_{2}, \pr_{\cB}^{2})$ is the canonical map $(0,Y,0) \to (F,Y,  F \to uv_*(Y))$. Using Lemma~\ref{lemma:equiv-of-oriented-pb}, we find that the counit of the adjunction $(\Phi^{-1}j_{2}, \pr_{\cB}^{1})$ on an object $(Y, X,  f\colon v_{*}(Y) \to u_{*}(X))$ is the canonical map
\[
(0, Y, 0) \lto ( \fib(\hat f), Y, \fib(\hat f) \stackrel{\mathrm{can}}{\lto} uv_{*}(Y) ),
\]
where as before $\hat f$ is the adjoint map of $f$. Applying $\Sigma j_{1*}$, which, as we recall, is the cofibre functor on the arrow component of the oriented fibre product, to this natural transformation gives the canonical map
$uv_{*}(Y) \to \cof ( \can)$. This map naturally identifies with the map $\hat f\colon uv_{*}(Y) \to X$. Thus, applying $u_{*}$ and precomposing with the unit of $(u,u_{*})$ we obtain the map $f$. Thus we finally see that the composite \eqref{eq:composite} that we wish to calculate and that equals \eqref{eq:second-composite} 
is the canonical transformation which to an object $(X,Y,f\colon v_*(Y) \to u_*(X))$ associates the morphism $f$. The claim that the resulting square obtained by restricting along the inclusion
\[ 
\cB \widetilde{\laxtimes{v_*,u_*}} \cA' \subseteq \cB \laxtimes{v_*,u_*} \cA' ,
\]
where the tilde denotes the full subcategory on objects where the morphism $f\colon v_*(Y) \to u_*(X)$ is an equivalence, is a pullback follows immediately.
\end{proof}

\begin{lemma}
	\label{lemma:right-adjoint-of-functor-from-A0}
The right adjoint of the functor $i_{0}\colon \cA_{0} \to \cA' \laxtimes{\id,uv_{*}} \cB$ is given by
\[
s_{0} \colon (X,Y, X \xrightarrow{f} uv_{*} (Y)) \mapsto u_{*}(X) \times_{u_{*}uv_{*}(Y)} v_{*}(Y),
\]
where the map $u_{*}(X) \to u_{*}uv_{*}(Y)$ is $u_{*}(f)$ and the map $v_{*}(Y) \to u_{*}uv_{*}(Y)$ is the unit of the adjunction $(u,u_{*})$.
\end{lemma}
\begin{proof}
Consider the map induced by the unit transformations of $(u,u_{*})$ and $(v,v_{*})$:
\[
X \lto u_{*}u(X) \times_{u_{*}uv_{*}v(X)} v_{*}v(X) = s_{0}i_{0}(X).
\]
From the formula for mapping spaces in the oriented fibre product, one reads off the desired equivalence
\[
\Map_{\slax}( i_0(X), (Y,Z, f) ) \simeq \Map_{ \cA_{0}} (X, s_{0}(Y,Z,f) )
\]
as indicated in the following pullback diagram:
\[
\begin{tikzcd}[column sep = 0cm]
\Map_{\slax}( (u(X), v(X), u(X) \to uv_{*}v(X)), (Y,Z, f) ) \ar[d]\ar[r] & \Map_{\cB}(v(X), Z) \simeq \Map_{\cA_{0}}(X, v_{*}(Z)) \ar[d]\\ 
\Map_{\cA_{0}}(X, u_{*}(Y)) \simeq \Map_{\cA'}(u(X), Y) \ar[r] & \Map_{\cA'}(u(X), uv_{*}(Z)) \simeq \Map_{\cA_{0}}(X, u_{*}uv_{*}(Z)) 
\end{tikzcd}
\]
\end{proof}

\begin{proof}[Proof of Theorem~\ref{thm:pushout-in-Cat-perf}]
As indicated earlier, the functor $\Ind\colon \Cat^k \to \Pr^{\L}_{\omega}(k)$ is an equivalence \cite[Prop.~4.5 \& 4.9]{HoyoisScherotzkeSibilla}, so we may equivalently show that the diagram of Ind-completions 
\[ \begin{tikzcd}
	\cA_0 \ar[r,"v"] \ar[d,"u"] & \cB \ar[d,"pj_2"] \\
	\cA' \ar[r,"p\Omega j_1"] & \cQ
\end{tikzcd}\]
is a pushout diagram in $\Pr_{\omega}^{\L}(k)$. Note that given a functor $\varphi\colon \cQ \to \cT$ in $\Pr^{\L}(k)$ such that both restrictions to $\cA'$ and $\cB$ preserve compact objects, then so does the functor $\varphi$. This is because the compact objects of $\cQ$ are generated (as a thick subcategory) by the images of compact objects of $\cA'$ and $\cB$, and compact objects in $\cT$ form a thick subcategory of $\cT$. Therefore, it suffices to show that the above square is a pushout in $\Pr^{\L}(k)$. 
As $\Pr^{\L}_{\st}$ is closed symmetric monoidal \cite[Rem.~4.8.1.18]{HA}, the tensor product commutes with colimits separately in each variable. Hence the forgetful functor $\Pr^{\L}(k) = \Mod_{\Mod(k)}(\Pr^{\L}_{\st}) \to \Pr^{\L}_{\st}$ preserves colimits \cite[Cor.~4.2.3.5]{HA}. As this forgetful functor is also conservative \cite[Cor.~4.2.3.2]{HA}, it is enough to prove that the above square is a pushout diagram in $\Pr^{\L}_{\st}$. 
Using the equivalence $\Pr^{\L}_{\st} \simeq (\Pr^{\mathrm{R}}_{\st})^{\op}$ (which fixes objects and sends a left adjoint functor to its right adjoint; cf.~\cite[Cor.~5.5.3.4]{HTT}), we may equivalently show that the diagram of right adjoints 
\[\begin{tikzcd}
	\cQ \ar[r,"\pr_{\cB}p_*"] \ar[d,"\Sigma j_{1*}p_*"'] & \cB \ar[d,"v_*"] \\
	\cA' \ar[r,"u_*"] & \cA_0
\end{tikzcd}\]
is a pullback diagram in $\Pr^{\mathrm{R}}_{\st}$. As limits in $\Pr^{\mathrm{R}}_{\st}$
 are computed in $\widehat{\Cat}_\infty$ (see~\cite[Thm.~4.4.3.18]{HTT} and \cite[Thm.~1.1.4.4]{HA}), it is finally enough to prove that the above square is a pullback diagram of $\infty$-categories.

Recall that $\cQ$ is defined as the Verdier quotient $(\cA' \laxtimes{\id, uv_{*}} \cB) / \cA$, where $\cA$ is the localizing subcategory of the oriented fibre product generated by the image of the canonical functor $i_{0} \colon \cA_0 \to \cA' \laxtimes{\id,uv_*} \cB$. The inclusion 
\[
 \cA  \hookrightarrow \cA' \laxtimes{\id,uv_{*}}  \cB
\]
admits a right adjoint $s$, and it is a general fact that $p_{*}$ identifies $\cQ$ with the subcategory $\ker(s) \subseteq  \cA' \laxtimes{\id,uv_{*}}  \cB$. 
By construction, the essential image of the functor $\cA_{0} \to \cA$ induced by $i_{0}$ generates the target under colimits. We deduce that its right adjoint is conservative, hence the kernel of $s$ coincides with the kernel of the right adjoint $s_{0}$ of $i_{0}$. 
We show below that the equivalence $\Phi$ from Lemma~\ref{lemma:equiv-of-oriented-pb} restricts to an equivalence
\begin{equation}
	\label{eq1}
 \cB \laxtimes{v_{*},u_{*}} \cA' \supseteq \cB \times_{v_{*},u_{*}} \cA'  \overset{\simeq}{\lto}   \ker(s_{0})   \subseteq \cA' \laxtimes{\id, uv_{*}}  \cB 
\end{equation}
where we view the pullback on the left-hand side as the full subcategory of the oriented fibre product spanned by the objects $(Y,X,f)$ with $f$ an equivalence. Together with Lemma~\ref{claim1} this proves the assertion.

Now we prove \eqref{eq1}. Obviously, $\Phi$ restricts to an equivalence $\ker(s_{0}\circ\Phi) \stackrel{\simeq}{\to} \ker(s_{0})$. It therefore suffices to determine $\ker(s_0\circ \Phi)$. 
So let $(Y,X,v_*(Y) \stackrel{f}{\to} u_*(X))$ be an object of $ \cB \laxtimes{v_{*},u_{*}} \cA'$. By the formula for $\Phi$ in Lemma~\ref{lemma:equiv-of-oriented-pb} and by
Lemma~\ref{lemma:right-adjoint-of-functor-from-A0} we then have a solid pullback diagram together with the dashed commutative diagram, making the lower horizontal composite a fibre sequence:
\[
\begin{tikzcd}
 s_0(\Phi(Y,X,f)) \ar[d]\ar[r] & v_{*}(Y) \ar[d, "\epsilon_{u}(v_{*}(Y))"] \ar[dr, dashed, bend left,"f"] \\ 
 u_{*}(\fib(\hat f)) \ar[r, "u_{*}(\can)"] & u_{*}uv_{*}(Y) \ar[r,dashed,"u_*(\hat{f})"] & u_*(X)
\end{tikzcd}
\]
As before, $\hat{f}\colon uv_{*}(Y) \to X$ is the adjoint map of $f$. We deduce a canonical equivalence $s_0(\Phi(Y,X,f)) \simeq \fib(f)$. Thus $(Y,X,f)$ belongs to $\ker(s_{0}\circ\Phi)$ if and only of $f$ is an equivalence. This proves the remaining claim.
\end{proof}

\begin{cor}\label{cor:categorical-thmA}
Let 
\[
\begin{tikzcd}
 \cA_{0} \ar[d, "u"']\ar[r, "v"] & \cB \ar[d] \\ 
 \cA' \ar[r] & \cQ 
\end{tikzcd}
\]
be a pushout square of small, $k$-linear categories. As before, let $\cA \subseteq \cA' \laxtimes{\id, uv_{*}} \cB$ be the thick subcategory generated by the image of the canonical functor $\cA_{0} \to \cA' \laxtimes{\id, uv_{*}} \cB$. 
Then the induced square
\[
\begin{tikzcd}
 \cA \ar[d]\ar[r] & \cB \ar[d] \\ 
 \cA' \ar[r] & \cQ 
\end{tikzcd}
\]
is a motivic pullback square.
\end{cor}

\begin{proof}
By Theorem~\ref{thm:pushout-in-Cat-perf} there is an equivalence $(\cA' \laxtimes{\id, uv_{*}} \cB)/\cA \simeq \cQ$ and arguing again as in \cite[p.~888]{LT} the resulting second square in the statement is a motivic pullback square.
\end{proof}

Finally, we record here the following consequence of Theorem~\ref{thm:pushout-in-Cat-perf}. To that end, let 
\begin{equation}\label{diagram:categories}
\begin{tikzcd}
	\cA \ar[r,"v"] \ar[d,"u"] & \cB \ar[d,"g"] \\
	\cA' \ar[r,"f"] & \cB' 
\end{tikzcd}
\end{equation}
be a commutative diagram in $\Cat_\infty^k$. Associated to the square \eqref{diagram:categories}, there is the canonical exchange transformation $\alpha\colon uv_* \to f_* g$. We say that \eqref{diagram:categories} is \emph{Ind-adjointable} if $\alpha$ is an equivalence.
\begin{cor}
Any Ind-adjointable bicartesian square \eqref{diagram:categories} is a motivic pullback square.
\end{cor}
\begin{proof}
Postcomposition with $\alpha$ induces a functor $\cA' \laxtimes{uv_*} \cB \to \cA' \laxtimes{g_*f} \cB \simeq \cA' \laxtimes{\cB'} \cB$ which participates in the following commutative diagram.
\[\begin{tikzcd}
	\cA \ar[r,"i"] \ar[d] & \cA' \laxtimes{uv_*} \cB \ar[d] \\
	\cA' \times_{\cA} \cB \ar[r] & \cA' \laxtimes{\cB'} \cB 
\end{tikzcd}\]
As a result of Ind-adjointability, the right vertical functor is an equivalence, and since \eqref{diagram:categories} is cartesian, the left vertical map is an equivalence. Consequently, $i$ is fully faithful and the result follows then from Corollary~\ref{cor:categorical-thmA} applied to the bicartesian square \eqref{diagram:categories}.
\end{proof}

\begin{rem}\label{rem:adjointable-squares}
Let us consider the following examples.
\begin{enumerate}
\item Consider a bicartesian square where one entry is $0$. Such squares are adjointable, and are motivic pullback squares either by additivity of localizing invariants (the case where $\cA$ or $\cB'$ is $0$) or by definition (when $\cA' = 0$, the square reduces to a bifibre sequence).
\item Consider a cartesian square where two parallel functors are Karoubi projections, that is, such that the fibre sequences obtained by adding the kernel of these functors are also cofibre sequences. By pushout pasting, any such square is cocartesian and a motivic pullback square as a consequence of the definitions of localizing invariants. Moreover, any such square is Ind-adjointable as follows e.g.\ from \cite[Prop.\ A.1.18]{CDHII}.
\item Dually, consider a cocartesian square where two parallel functors are fully faithful. Expanding out this square to include the Verdier quotients of the respective subcategories, one finds that the resulting sequences are bifibre sequences. Again, it follows that such a square is cartesian and a motivic pullback square, as well as Ind-adjointable as follows again from \cite[Prop.\ A.1.18]{CDHII}.
\end{enumerate}

Finally, we note that \cite[Theorem 3.4.3]{BachmannCatMilnor} also gives examples of motivic pullback squares, namely adjointable Milnor squares\footnote{What we call adjointable is called \emph{satisfying base-change} in loc.\ cit.}. By definition, these are certain squares of compactly generated stable $\infty$-categories, and combining the above arguments with the arguments for the proof of \cite[Theorem 3.4.3]{BachmannCatMilnor} it follows that adjointable Milnor squares give rise to Ind-adjointable bicartesian squares in $\Cat_\infty^\perf$ upon passing to compact objects.
\end{rem}

\subsection{The $\odot$-ring as a pushout}
	\label{sec:pushout-rings}

We now prove Theorem \ref{ThmB} from the introduction. 
So assume given a Milnor context $(A',B,M)$, and let $A = A' \boxtimes_{M} B$ be the $k$-algebra given by Theorem~\ref{thm:GLT} (or Construction~\ref{dfn:lambda-and-A}). Recall that a tensorizer for $(A',B,M)$ consists of $k$-algebra maps $A_0 \to A'$ and $A_0 \to B$ together with a commutative diagram 
\[ \begin{tikzcd}
	A_0 \ar[r] \ar[d] & B \ar[d] \\
	A' \ar[r] & M
\end{tikzcd}\]
of $(A_0,A_0)$-bimodules, such that the induced $(B,A')$-bimodule map $B\otimes_{A_0} A' \to M$ is an equivalence. Unravelling the definitions, a tensorizer is equivalently described by a morphism of Milnor contexts $(A_{0},A_{0},A_{0}) \to (A',B,M)$ such that the induced map $B\otimes_{A_{0}} A' \to M$ is an equivalence, see Definition~\ref{def:map-of-milnor-contexts}. By Lemma~\ref{lemma:map-of-milnor-contexts}, a tensorizer induces a map $A_{0} \to A$ of $k$-algebras.
By abuse of notation, we refer to $A_{0}$ itself as a tensorizer.
For convenience, we state here again Theorem~\ref{ThmB} from the introduction.

\begin{thm}\label{thm:thm-B-im-Text}
Let $(A',B,M)$ be a Milnor context with tensorizer $A_{0}$. Then, the canonical commutative square 
\[
\begin{tikzcd}
	A_0 \ar[r] \ar[d] & B \ar[d] \\
	A' \ar[r] & A' \wtimes{A}{M} B
\end{tikzcd}
\]
obtained from Theorem~\ref{thm:GLT} and the map $A_{0} \to A$ is a pushout diagram in $\Alg(k)$.
\end{thm}
\begin{proof}
We first observe that the functor $\Perf(-)\colon \Alg(k) \to \Cat^k_{\infty}$ factors as the composite $\Alg(k) \to (\Cat^k_{\infty})_{\ast/} \to \Cat^k_{\infty}$, where the first functor sends $A$ to the pair $(\Perf(A),A)$. The first functor is a fully faithful left adjoint, its right adjoint is given by taking $(\cC,c)$ to $\End_\cC(c)$. The second functor preserves contractible colimits and is conservative. In total, the composite $\Perf(-)\colon \Alg(k) \to \Cat^k_{\infty}$ is conservative and preserves pushouts. Therefore, it suffices to prove that the diagram 
\[ \begin{tikzcd}
	\Perf(A_0) \ar[r] \ar[d] & \Perf(B) \ar[d] \\
	\Perf(A') \ar[r] & \Perf(A' \wtimes{A}{M} B) 
\end{tikzcd}\]
is a pushout. To do so, we consider the diagram as constructed in Secton~\ref{sec:bimodule-construction}:
\[\begin{tikzcd}
	\Perf(A_0) \ar[dr] \ar[ddr, bend right ] \ar[drr, bend left] & & \\ 
		& \Perf(A) \ar[r] \ar[d] & \Perf(B) \ar[d] \\
		& \Perf(A') \ar[r] & \Perf(A') \laxtimes{M} \Perf(B)
\end{tikzcd}\]
Here, we use the assumption that $A_0 \to A$ is a tensorizer to see that the oriented fibre product appearing here is indeed the one associated to the span $\Perf(A') \leftarrow \Perf(A_0) \to \Perf(B)$ in Section~\ref{sec:bimodule-construction}, see Example~\ref{example:span-tensorizer}. Theorem~\ref{thm:pushout-in-Cat-perf} shows that the big square in the above diagram induces the following pushout in $\Cat^k_\infty$
\[\begin{tikzcd}
	\Perf(A_0) \ar[r] \ar[d] & \Perf(B) \ar[d] \\
	\Perf(A') \ar[r] & \big[\Perf(A') \laxtimes{M} \Perf(B)\big]/\Perf(A)
\end{tikzcd}\]
Finally, we use that $\big[\Perf(A') \laxtimes{M} \Perf(B)\big]/\Perf(A) \simeq \Perf(A' \wtimes{A}{M} B)$ and under this equivalence, the above square is the one we wish to show is a pushout.
\end{proof}

\subsection{Some complements on tensorizers and the $\odot$-ring}

In this subsection we briefly discuss existence and uniqueness of tensorizers and comment on the asymmetry of the $\wtimes{}{}$-ring.

\begin{ex}[Tensorizers need not exist]
In general, a tensorizer for a given Milnor context need not exist. Indeed, we claim that the Milnor context provided by the pullback square 
\[ \begin{tikzcd}
	\Z \ar[r] \ar[d] & \Z[y] \ar[d] \\
	\Z[x] \ar[r] & \Z[t^{\pm 1}],	
\end{tikzcd}
\]
where the maps send $x$ to  $t$ and $y$ to $t^{-1}$, does not admit a tensorizer.
\end{ex}
\begin{proof}
Aiming for a contradiction, suppose a tensorizer $A_0$ is given and let $A_{0} \to \Z$ be the induced map of $\SS$-algebras. We then consider the composite 
\[ 
\Z[x] \otimes_\SS \Z[y] \to \Z[x] \otimes_{A_0} \Z[y] \to \Z[x] \otimes_\Z \Z[y] 
\]
induced by the maps $\SS \to A_0 \to \Z$. It is easy to see that the displayed composite is an isomorphism on $\pi_0$ because the map $\SS \to \Z$ is. Furthermore, by assumption the induced map $\Z[x] \otimes_{A_{0}} \Z[y] \to \Z[t^{\pm 1}]$ is an equivalence, and under this equivalence, the effect on $\pi_0$ of the first map is given by the canonical composite $\Z[x,y] \to \Z[t^{\pm 1}]$ sending $x$ to $t$ and $y$ to $t^{-1}$. This map is not injective, contradicting the above composite to be an isomorphism on $\pi_0$.
\end{proof}

\begin{rem}
It was pointed out by an anonymous referee that in the situation above, a \emph{categorical tensorizer} does exist, i.e.\ a small stable $\infty$-category $\cA_0$ fitting into an Ind-adjointable commutative diagram as follows.
\[ \begin{tikzcd}
	\cA_0 \ar[r] \ar[d] & \Perf(\Z[y]) \ar[d] \\
	\Perf(\Z[x]) \ar[r] & \Perf(\Z[t^{\pm1}])
\end{tikzcd}\]
Indeed, one can simply take $\cA_0$ to be the pullback, i.e.\ $\Perf(\P^1_\Z)$; that this results in an Ind-adjointable square is again a consequence of \cite[Prop.\ A.1.18]{CDHII}. As argued in Remark~\ref{rem:adjointable-squares}, choosing $\cA_0=\Perf(\P^1_\Z)$ results is a cocartesian square. 
On the other hand, in \cite[Prop.~4.1]{LT} we computed the $\odot$-ring in this situation to be $\Z\langle x,y\rangle/(yx-1)$ whose category of perfect modules is not equivalent to $\Perf(\Z[t^{\pm 1}])$ (as follows, e.g., from Cor.~4.2 there).
This implies that even when a categorical tensorizer for a pullback square of $k$-algebras exists, $\Perf(\wtimes{}{})$ is not in general the pushout over the tensorizer. 
\end{rem}

\begin{ex}[Tensorizers are not unique]
Assume that $A_{0}$ is a tensorizer for a given Milnor context $(A', B, M)$. We claim that if $C \to A_{0}$ is a homological epimorphism\footnote{Such maps were called Tor-unital in \cite{Tamme:excision} and \cite{LT}.}, i.e.~the multiplication map $A_{0} \otimes_{C} A_{0} \to A_{0}$ is an equivalence, then also $C$ is a tensorizer for that Milnor context. Indeed, we simply compute
\[
B \otimes_{C} A' \simeq B \otimes_{A_{0}} A_{0} \otimes_{C}  A_{0} \otimes_{A_{0}} A' \simeq B \otimes_{A_{0}} A_{0} \otimes_{A_{0}} A' \simeq B \otimes_{A_{0}} A'.
\]
As another example, consider a discrete commutative ring $k$. Then as $(k,k)$-bimodule we have
\[
k \otimes_{k[x,y]} k \simeq k \oplus \Sigma k^{\oplus 2} \oplus \Sigma^{2} k
\]
so that $k[x,y]$ is a tensorizer for the Milnor context $(k,k, M)$, where $M =  k \oplus \Sigma k^{\oplus 2} \oplus \Sigma^{2} k$. On the other hand, we can consider the tensor algebra $T_{k}(N)$ on the $(k,k)$-bimodule $N = k^{\oplus 2} \oplus \Sigma k$. Then 
Lemma~\ref{lem:tensor-product-over-tensor-alg} in the next section gives a canonical equivalence
\[
k \otimes_{T_{k}(N)} k \simeq k \oplus \Sigma N \simeq M
\]
so that $T_{k}(N)$ is another tensorizer for the Milnor context $(k,k,M)$. In addition, one can check that the canonical map $T_{k}(N) \to k[x,y]$ is not a homological epimorphism. Indeed, suppose it were. Then we find that
\[ k[x,y] \otimes_{T_k(N)} k \simeq \big( k[x,y] \otimes_{T_k(N)} k[x,y] \big) \otimes_{k[x,y]} k \simeq k.\]
Now we consider the fibre sequence
\[ T_k(N) \otimes_k N \lto T_k(N) \lto k \]
and apply $k[x,y] \otimes_{T_k(N)} (-)$ to obtain the fibre sequence
\[ k[x,y] \otimes_k N \lto k[x,y] \lto k \]
which is in contradiction to the fact that $\pi_1(k[x,y] \otimes_k N) \cong k[x,y] \neq 0$.
\end{ex}

\begin{rem}
In general, the $k$-algebra $A' \wtimes{A}{B'}B$ associated to a pullback square of $k$-algebras in \cite{LT} depends on the chosen orientation of the oriented fibre product. In other words, in general, as far as we understand there is no canonical comparison between the rings
\[ A' \wtimes{A}{B'} B \quad \text{ and} \quad B \wtimes{A}{B'} A'.\]
This becomes more apparent in the generality of Milnor contexts: given one such $(A',B,M)$, the symbol $B \wtimes{A}{M} A'$ does not even make sense. However, let us assume that the base ring $k$ is $\E_\infty$, and assume given a Milnor context $(A',B,B')$ associated to a pullback diagram of $\E_2$-$k$-algebras and a tensorizer $A_0$ for which the map $A_{0} \to A$ is also $\E_2$. Then the two maps of spectra
\[ 
B \otimes_{A_0} A' \to B' \quad \text{ and } \quad A' \otimes_{A_0} B \to B' 
\] 
are equivalent, and hence $A_{0}$ is also tensorizer for the Milnor context $(B, A', B')$.
As a result, in this case we deduce the equivalences
\[ 
A' \wtimes{A}{B'} B \simeq A' \amalg_{A_0} B \simeq B \amalg_{A_0} A' \simeq B \wtimes{A}{B'} A'
\]
by Theorem~\ref{ThmB} and the symmetry of the pushout.
\end{rem}

\section{Examples and Applications}
	\label{sec:examples}

\subsection{Free algebras and trivial square zero extensions.} 
In this section we describe, for a localizing invariant $E$, a general relation between the $E$-theory of free algebras and trivial square zero extensions. We fix a $k$-algebra $C$. Recall that there is a functor $\Alg(k)_{C/} \to \BMod(C,C)$ which takes $C \to D$ to $D$, viewed as a $(C,C)$-bimodule. This functor has a left adjoint given by the tensor algebra $T_C(-)$: 
\[ \BMod(C,C) \lto \Alg(k)_{C/} \, , \quad M \mapsto T_C(M) = \bigoplus_{n\geq 0} M^{\otimes_C^n} .\]
On the other hand, one can stabilize the $\infty$-category $\Alg(k)_{/C}$. This stabilization turns out to be equivalent to $\BMod(C,C)$ \cite[paragraph before Remark 7.4.1.12]{HA}, and comes with a right adjoint forgetful functor 
\[ \BMod(C,C) \lto \Alg(k)_{/C}\quad,\quad M \mapsto C\oplus M \]
where the latter is called the trivial square zero extension of $C$ by $M$. We observe that the map $0 \to M$ of $(C,C)$-bimodules provides a map $C \to C\oplus M$ in $\Alg(k)_{/C}$, so the trivial square zero extension is in particular canonically a $(C,C)$-bimodule. As such, it is in fact equivalent to $C \oplus M$, so we allow ourselves the abuse of notation writing also $C \oplus M$ for the ring.

\begin{thm}\label{thm:free-vs-sz}
Let $C$ be a $k$-algebra and $M$ a $(C,C)$-bimodule. Then there is a canonical commutative square of $k$-algebras
\[ \begin{tikzcd}
	C \oplus \Omega M \ar[r] \ar[d] & C \ar[d] \\
	C \ar[r] & T_C(M) \, .
\end{tikzcd}
\]
This square is a motivic pullback square.
\end{thm}
\begin{proof}
We consider the evident Milnor context $(C,C,C\oplus M)$ where $C\oplus M$ is pointed via the inclusion of the first summand, and identify the terms of the motivic pullback square produced by Theorem~\ref{thm:GLT}. By Lemma~\ref{lem:pullback-bimodule} below, the associated $k$-algebra $C \boxtimes_{C\oplus M}C$ is the trivial square zero extension $C \oplus \Omega M$. It remains to identify the $\odot$-ring for this Milnor context. According to Lemma~\ref{lem:tensor-product-over-tensor-alg} below, this Milnor context has a tensorizer, namely $T_{C}(\Omega M)$ with twice the augmentation. Theorem~\ref{ThmB} thus implies that the $\odot$-ring in question is the pushout $C \amalg_{T_{C}(\Omega M)} C$ in $\Alg(k)$. As the functor $T_{C}(-)\colon \BMod(C,C) \to \Alg(k)_{C/}$ is a left adjoint and the forgetful functor $\Alg(k)_{C/} \to \Alg(k)$ preserves contractible colimits, we find that that $C \amalg_{T_{C}(\Omega M)} C$ is canonically equivalent to $T_{C}(0 \amalg_{\Omega M} 0) \simeq T_{C}(M)$. This finishes the proof of the theorem.
\end{proof}

\begin{lemma}
	\label{lem:pullback-bimodule}
Let $C$ be a $k$-algebra and $M$ a $(C,C)$-bimodule. Consider the Minor context $(C,C,C\oplus M)$, where $C \oplus M$ is pointed via the inclusion of the first summand. Then its associated $k$-algebra $C \boxtimes_{C\oplus M}C$ is the trivial square zero extension $C \oplus \Omega M$ of $C$.
\end{lemma}

\begin{proof}
We view the $(C,C)$-bimodule $C \oplus M$ as a $k$-algebra via the trivial square zero extension. Then the cospan $C \to C \oplus M \leftarrow C$ induced by the base point is canonically a cospan of $k$-algebras: It arises by applying the trivial square zero extension functor to the cospan $0\to M \leftarrow 0$. According to Remark \ref{rem:comparison-LT1}, the $k$-algebra $C \boxtimes_{C\oplus M}C$  is then the pullback of the cospan $C \to C \oplus M \leftarrow C$ in $\Alg(k)$. As the trivial square zero extension functor $C \oplus (-) \colon\BMod(C,C) \to \Alg(k)_{/C}$ is a right adjoint and the forgetful functor $\Alg(k)_{/C} \to \Alg(k)$ preserves contractible limits, 
we find that the above pullback is given by the trivial square zero extension $C \oplus \Omega M$.
\end{proof}

\begin{lemma}
	\label{lem:tensor-product-over-tensor-alg}
Let $C$ be a $k$-algebra and $M$ a $(C,C)$-bimodule. Let $T_{C}(\Omega M) \to C$ be the canonical map of $k$-algebras induced by $\Omega M \to 0$. Then 
the inclusion $C \to C \oplus M$ induces an equivalence of $(C,C)$-bimodules
\[
C \otimes_{T_{C}(\Omega M)} C \simeq C \oplus M.
\]
\end{lemma}

\begin{proof}
We have the following fibre sequence of $(C,T_C(M))$-bimodules. 
\[ 
\Omega M \otimes_C T_C(\Omega M) \lto T_C(\Omega M) \lto C
\]
Applying $(-) \otimes_{T_C(\Omega M)} C$, we obtain a fibre sequence of $(C,C)$-bimodules 
\[ 
\Omega M \lto C \lto C \otimes_{T_C(\Omega M)} C
\]
and observe that the second map admits a $(C,C)$-bimodule retraction, namely the multiplication map. Consequently, we obtain the equivalence of $(C,C)$-bimodules as claimed.
\end{proof}

\subsubsection*{$E$-theory of endomorphisms}
For a $k$-localizing invariant $E$, we discuss the $E$-theory of (twisted, nilpotent) endomorphisms, and its relation with trivial square zero extensions and hence with the $E$-theory of tensor algebras.

\begin{dfn}
Let $C$ be a $k$-algebra, and let $M$ be a $(C,C)$-bimodule. Let $E$ be a $k$-localizing invariant. We define
\[
\Nil E(C; M) := \cof( E(C) \to E(C\oplus {\Omega}M)), \qquad NE(C;M) := \cof(E(C) \to E(T_{C}(M)).
\]
\end{dfn}
As $C$ is a canonical retract of the trivial square zero extension $C \oplus M$ and of the tensor algebra $T_{C}(M)$, we obtain canonical decompositions
\[
E(C\oplus {\Omega}M) \simeq E(C) \oplus \Nil E(C;M), \qquad E(T_{C}(M)) \simeq E(C) \oplus NE(C;M).
\]
In this notation, $NK(C;C)$ is the usual $NK$-spectrum of $C$ which measures the failure of $\mathbb{A}^1$-homotopy invariance of $K$-theory. The notation $\Nil E$ refers to its relation to the $E$-theory of nilpotent endomorphisms, as we explain below. This will also clarify the appearance of the shift $\Omega$ in the definition. But first, we record an immediate consequence of Theorem~\ref{thm:free-vs-sz} in this notation.

\begin{cor}\label{cor:Nil-vs-N}
Let $C$ be a $k$-algebra and $M$ a $(C,C)$-bimodule. Then there is a canonical equivalence $\Sigma\Nil E(C; M) \simeq NE(C;M)$.
\end{cor}

\begin{dfn}
Let $C$ be a $k$-algebra and $M$ a $(C,C)$-bimodule. We define the presentable $k$-linear $\infty$-category of $M$-twisted endomorphisms of $C$-modules $\underline\End(\RMod(C); M)$ as the pullback of the diagram
\[
\begin{tikzcd}[column sep=2cm]
\RMod(C) \ar[r, "{(\id, (-)\otimes_{C}M)}"] & \RMod(C) \times \RMod(C) & \RMod(C)^{\Delta^{1}}. \ar[l, "{(s, t)} "']
\end{tikzcd}
\]
Thus its objects are pairs $(X, f)$ consisting of a $C$-module $X$ and a $C$-linear morphism $f\colon X \to X \otimes_{C} M$.
Likewise, we let $\underline{\End}(\Perf(C);M)$, or simply $\underline{\End}(C;M)$, be the full subcategory of $\underline{\End}(\RMod(C);M)$ consisting of those pairs $(X,f)$ for which $X$ is perfect. In particular, $\underline{\End}(C;M)$ is a full subcategory of $\underline{\End}(\RMod(C);M)$ consisting of compact objects.

We then define the subcategory of nilpotent $M$-twisted endomorphisms $\underline\Nil (\Perf(C); M)$, or simply $\underline{\Nil}(C;M)$, as the thick subcategory of $\underline\End(C;M)$ generated by the object $(C, C \xrightarrow{0} M)$. 
\end{dfn}

\begin{prop}\label{prop:twisted-nilpotent}
Let $C$ be a $k$-algebra and $M$ a $(C,C)$-bimodule. There is a canonical equivalence
$
\underline\Nil(C; M) \simeq \Perf(C \oplus \Omega M).
$
In particular, for any $k$-localizing invariant $E$, there is a canonical equivalence
\[
E(\underline\Nil(C;M)) \simeq E(C) \oplus \Nil E(C; M).
\]
\end{prop}
\begin{proof}
By definition, $\underline\Nil(C;M)$ is generated by the object $(C,0)=(C,C\xrightarrow{0}M)$. Thus $\underline\Nil(C;M)$ identifies  with the $\infty$-category of perfect $\End((C,0))$-modules. To compute $\End((C,0))$, we observe that $\underline\End(C;M)$ is equivalent to the full subcategory of the oriented fibre product
\[ 
\Perf(C) \laxtimes{C \oplus M} \Perf(C) 
\]
on those objects $(X,Y, X \stackrel{(f,g)}{\to} Y\oplus (Y\otimes_C M))$ where the map $f\colon X \to Y$ is an equivalence. 
Under this equivalence, $(C,0)$ corresponds to the object $\Lambda = (C, C, C \xrightarrow{(\id, 0)} C \oplus M)$ of Construction~\ref{dfn:lambda-and-A}. Its endomorphism ring is  the pullback $C \times_{C\oplus M} C$, which by Lemma~\ref{lem:pullback-bimodule} is given by the trivial square zero extension $C \oplus \Omega M$, as desired.
\end{proof}

\begin{ex}
In the above setting, assume that $C$ is a connective $k$-algebra, and that the bimodule $M$ is connected. Then we claim that  $\underline\Nil(C;M)=\underline\End(C;M)$. 
To see this, it suffices to check that $(C,0)$ generates $\underline\End(C;M)$ as an idempotent complete stable $\infty$-category, or equivalently, that the functor on the Ind-completion $\Ind(\underline\End(C;M))$ corepresented by $(C,0)$ is conservative. As $\underline\End(C;M)$ consists of compact objects in $\underline\End(\RMod(C); M)$, this Ind-completion canonically identifies with a full subcategory of $\underline\End(\RMod(C); M)$.
The functor corepresented by $(C,0)$ is the functor $\fib$ which sends an object $(X, f\colon X \to X \otimes_{C} M)$ to $\fib(f)$. 
As $\Ind(\underline\End(C;M))$ is stable, it suffices to check that $\fib$ detects zero objects.  So consider an object $(X,f)$ of $\Ind(\underline\End(C;M))$ and assume that $\fib(f) \simeq 0$, i.e.~that $f$ is an equivalence. Fix an integer $n$, and let $x$ be an arbitrary element of $\pi_{n}(X)$. We have to show that $x=0$.
Write $(X,f)$ as the colimit of diagram $(X_{i}, f_{i})_{i\in I}$ with a small filtered $\infty$-category $I$.
Then there exists an $i \in I$ such that $x$ has a preimage $x_{i}$ in $\pi_{n}(X_{i})$. As $C$ is connective and $X_{i}$ is a perfect $C$-module, $X_{i}$ is bounded below. As $M$ is connected, it follows that there exists a positive integer $k$ such that $X_{i} \otimes_{C} M^{\otimes k}$ is $n$-connected. Consider the following commutative diagram.
\[
\begin{tikzcd}[column sep=1.6cm]
X_{i} \ar[d] \ar[r, "f_{i}"] & X_{i}\otimes_{C} M \ar[d] \ar[r, "f_{i} \otimes \id_{M}"] & \cdots \ar[d]\ar[r] & X_{i} \otimes_{C} M^{\otimes k} \ar[d]  \\
X \ar[r, "f"] & X\otimes_{C} M\ar[r, "f\otimes \id_{M}"] & \cdots \ar[r] & X \otimes_{C} M^{\otimes k}
\end{tikzcd}
\]
By choice of $k$, the top horizontal composite sends $x_{i}$ to $0$, and hence the lower horizontal composite sends $x$ to $0$. As by assumption all the lower horizontal maps are equivalences, it follows that $x=0$, as desired.

As a consequence, in this setting we have an equivalence
\[
\End E( C;M ) \simeq \Omega NE(C;M)
\]
where $\End E$ is defined similarly as $\Nil E$ before. 
When $E$ is $K$-theory, and $C$ is discrete, this is closely related to a result of Lindenstrauss and McCarthy \cite[Cor.~3.3]{LM2}, where they prove the analog result for $K$-theory, but work with exact categories of finitely generated projective modules and $M$-twisted endomorphisms, see Remark~\ref{rem:exact-vs-stable-concrete-case} for a comparison. Their approach is based on Goodwillie calculus, viewing both sides of the above display as functors in the simplicial bimodule $M$. The fact that $M$ is connected is used in \cite{LM2} to ensure that the respective Goodwillie towers converge.
 \end{ex}

\begin{rem}\label{rem:exact-vs-stable-concrete-case}
Classically, for a discrete ring $C$ and a discrete bimodule $M$, the following is considered in the literature, e.g.\ in \cite{LM2, DKR, Waldhausen2}. Namely the category $\underline{\End}(\Proj(C);M)$ consisting of pairs $(P,f \colon P \to P \otimes_C M)$ where $P$ is a finitely generated projective right $C$-module. This is an exact category in the sense of Quillen with admissible monomorphisms the maps which are direct summand inclusions on the underlying projective modules and admissible epimorphisms the maps which are surjections on the underlying module. The evident functor $\underline{\End}(\Proj(C);M) \to \underline{\End}(\Perf(C);M)$ is fully faithful and exact and its image generates the target as a stable $\infty$-category. This functor induces an equivalence on $K$-theory as can be extracted from the literature as follows. 

First, let $\mathscr{E}$ be an exact $\infty$-category in the sense of Barwick \cite{Barwick}, for instance an ordinary exact category in the sense of Quillen, see \cite[Example 3.3.1]{Barwick}, and denote by $\Cat_\infty^\mathrm{ex}$ the $\infty$-category of exact $\infty$-categories with exact functors. Any stable $\infty$-category is canonically an exact $\infty$-category, see \cite[Example 3.3.2]{Barwick}, and there is a fully faithful inclusion $\Cat_\infty^\mathrm{st} \subseteq \Cat_\infty^\mathrm{ex}$. 
This inclusion admits a left adjoint $\mathscr{E} \mapsto \mathrm{Stab}(\mathscr{E})$, see \cite{Klemenc}, where $\mathrm{Stab}(\mathscr{E})$ is the small stable subcategory of $\Fun^{\mathrm{ex}}(\mathscr{E}^\op,\Spc) \otimes \Sp$ generated by the Yoneda image. Here, the superscript ex refers to functors which send exact sequences in $\mathscr{E}$ to fibre sequences and the tensor product is the Lurie tensor product on $\Pr^\L$ \cite[\S 4.8.1]{HA}.
The unit of the above adjunction is a fully faithful and exact inclusion $\mathscr{E} \subseteq \mathrm{Stab}(\mathscr{E})$ induced by the Yoneda embedding of $\mathscr{E}$.

For exact categories $\mathscr{E}$ in the sense of Quillen, it follows from \cite[Corollary 7.59]{BCKW} that $\mathrm{Stab}(\mathscr{E})$ is equivalent to the $\infty$-category associated to the Waldhausen category $\mathrm{Ch}^\mathrm{b}(\mathscr{E})$ of bounded chain complexes in $\mathscr{E}$ with quasi-isomorphisms as weak equivalences. Combining the theorem of Gillet--Waldhausen \cite[Theorem 1.11.7]{ThomasonTrobaugh} and \cite[\S 8]{Barwick} or \cite[\S 7.2]{BGT} one finds that the map $K(\mathscr{E}) \to K(\mathrm{Stab}(\mathscr{E}))$ is an equivalence for any exact category in the sense of Quillen. It then suffices to observe that the fully faithful exact functor $\underline{\End}(\Proj(C);M) \to \underline{\End}(\Perf(C);M)$ mentioned above induces an exact functor 
\[ \mathrm{Stab}(\underline{\End}(\Proj(C);M)) \lto \underline{\End}(\Perf(C);M) \]
which is  fully faithful and essentially surjective (because its restriction to $\underline{\End}(\Proj(C);M)$ is fully faithful and the image of $\underline{\End}(\Proj(C);M)$ in both $\mathrm{Stab}(\underline{\End}(\Proj(C);M))$ and $\underline{\End}(\Perf(C);M)$ generates the respective stable $\infty$-categories).
\end{rem}

\begin{ex}\label{cor:dual-numbers}
We now specialize Theorem~\ref{thm:free-vs-sz} to obtain results about the ring of dual numbers $C[\epsilon] = C[x]/(x^2)$ over a discrete $k$-algebra $C$ (which is the trivial square-zero extension on the bimodule $C$). In this case, let us denote the tensor algebra $T_C(\Sigma C)$ by $C[t]$ with $|t|=1$. Then for any $k$-localizing invariant $E$, we obtain an equivalence
\[ E(C[\epsilon],C) \simeq \Omega E(C[t],C),\]
where the terms $E(C[\epsilon],C)$ and $E(C[t],C)$ denote the fibres of the the maps induced by the respective augmentations. In particular, we obtain that $K(C[\epsilon],C)$ is equivalent to $\Omega K(C[t],C)$.
\end{ex}

\subsubsection*{$K$-theory of homotopical polynomial algebras}
We view the ring spectra $C[t]$, for a discrete ring $C$ and arbitrary $|t|$, as fundamental objects and would like to evaluate their $K$-theory in general. We remark here already, that the case $|t|=2$ is related to the $K$-theory of the 2-dimensional coordinate axes over $C$, by Corollary~\ref{cor:coordinate-axes} below. It turns out that the general behaviour of the $K$-theories $K(C[t],C)$ with arbitrary $|t|$ obeys a certain dichotomy, namely the case where $|t|>0$ and the case where $|t|\leq 0$.
For instance, for $|t|>0$ we find that $K(C[t],C)$ is non-trivial in many cases\footnote{It follows for instance from \cite{Waldhausen} that $K_{|t|+1}(C[t],C) \cong \HH_0(C) \cong C/\langle ab-ba \;|\; a,b \in C \rangle$ which is often (but not always) non-trivial.}. In particular this is so when $C$ is regular Noetherian, in which case classical $\A^1$-invariance of $K$-theory gives $K(C[t],C) = 0$ for $|t|=0$. For $|t|<0$ the same result is true, for instance by \cite{BL} or by combining Theorem~\ref{thm:free-vs-sz} with \cite[Theorem 4.8]{AGH} applied to the cochain algebra $C^{*}(S^{1-|t|};C)$ and using that this algebra is, as an $\E_{1}$-algebra equivalent to the trivial square zero extension.
In the following, we will discuss some aspects of what is known about $K(C[t],C)$ for $|t|>0$ and begin with the following rational calculation.

\begin{prop}\label{eq:K-theory-of-homotopical-polynomial-algebras}
Let $C$ be a discrete ring and $\ell=|t|>0$, and denote by $C_\Q$ the rationalization of $C$. Then there is a canonical equivalence
\[ K(C[t],C)_\Q \simeq \begin{cases} \bigoplus\limits_{n \geq 1} \Sigma^{\ell(2n-1)+1}\HH(C_\Q)& \text{ if }  \ell=|t| \text{ is odd }\\ 
					\bigoplus\limits_{n\geq 1} \Sigma^{\ell n+1}\HH(C_\Q)& \text{ if } \ell= |t| \text{ is even.} \end{cases} \]
\end{prop}

If $C$ is a $\Q$-algebra, then $K(C[t],C)$ is rational, e.g.\ by \cite[Lemma 2.4]{LT}. If $C$ is in addition commutative regular Noetherian  (and hence ind-smooth over $\Q$ by \cite[Theorem 1.8]{Popescu}), one can evaluate $\HH(C)$ by means of the Hochschild--Kostant--Rosenberg theorem \cite[Theorem~3.4.4]{Loday} which in this case provides an isomorphism
\[ 
\Omega^{*}_{C/\Q} \stackrel{\cong}{\lto} \HH_*(C)
\]
so that relative algebraic $K$-theory $K(C[t],C)$ is described purely in terms of differential forms on $C$. As a direct consequence, one can for instance recover the ranks of the $K$-groups of dual numbers over rings of integers in number fields as initially calculated by Soul\'e \cite{Soule} who used group homological arguments.

\begin{proof}[Proof of Proposition~\ref{eq:K-theory-of-homotopical-polynomial-algebras}]
Goodwillie's theorem \cite{Goodwillie} provides a canonical equivalence 
\[ K(C[t],C)_\Q \simeq \Sigma \HC(C_\Q[t],C_\Q)\]
where $\HC$ denotes cyclic homology relative to $\Q$. It therefore suffices to calculate the right hand side, so we may assume that $C$ is rational. This relative cyclic homology is known to be as claimed, a very similar result appears for instance in \cite[Prop.~5.3.10]{Loday}. Since the argument is quick and easy, we provide a calculation nevertheless.
We recall that $\HC(-) = \HH(-)_{h\T}$, the orbits of the canonical circle action on Hochschild homology. We note that since $C$ is rational, we have a canonical equivalence $\THH(C) \simeq \HH(C)$ where the latter denotes Hochschild homology relative to $\Q$ and the former denotes (topological) Hochschild homology relative to $\SS$.
We have 
\begin{align*}
	\HH(C[t]) & \simeq \THH(\SS[\Omega \Sigma S^\ell]) \otimes \HH(C) \\
		& \simeq \SS[L(S^\ell)] \otimes \HH(C) \\
		& \simeq \bigoplus\limits_{n \geq 0} \big[S^1_+\wedge_{C_n} S^{\ell n} \big] \otimes \HH(C) \\
		& \simeq \bigoplus\limits_{n \geq 0} S^1_+ \wedge_{C_n} \big[ S^{\ell n} \otimes \HH(C) \big]
\end{align*}
where the $C_n$-action on $S^{\ell n} \otimes \HH(C)$ is the diagonal action, with $C_n$ acting by permutations on $S^{\ell n}$ and by the restricted $C_n$-action of the canonical $\T$-action on $\HH(C)$. Indeed, this follows from the fact that (topological) Hochschild homology is a symmetric monoidal functor, the calculation of topological Hochschild homology of spherical group rings, see e.g.\ \cite[Corollary IV.3.3]{NS}, and Goodwillie's stable splitting of free loop spaces, see e.g.\ \cite{Cohen, Bandklayder}.
We conclude that 
\[ \HC(C[t]) \simeq \HH(C)_{h\T} \oplus \bigoplus\limits_{n \geq 1} \big[S^{\ell n} \otimes \HH(C) \big]_{hC_n}.\]
We note that $S^{\ell n}\otimes \HH(C)$ is a rational spectrum, so its $C_n$-orbits are simply determined by the action of $C_n$ on its homotopy groups. The action on $\HH_*(C)$ is trivial (as it extends to an action of the connected group $\T$), and the generator of $C_n$ acts on $S^{\ell n}\otimes \Q$ by $(-1)^{\ell (n-1)}$. Consequently, we obtain the following calculation
\[
K(C[t],C) \simeq \begin{cases} \bigoplus\limits_{n \geq 1} \Sigma^{\ell(2n-1)+1}\HH(C)& \text{ if }  \ell=|t| \text{ is odd }\\ 
					\bigoplus\limits_{n\geq 1} \Sigma^{\ell n+1}\HH(C)& \text{ if } \ell= |t| \text{ is even.} \end{cases} 
\]
as claimed.
\end{proof}

\begin{rem}
The Dundas--Goodwillie--McCarthy theorem \cite{DGM} allows to calculate $K(C[t],C)$ for $|t|>0$ without rationalization, again by trace methods, extending the rational result of Goodwillie used above. Concretely, the cyclotomic trace induces an equivalence
\[ K(C[t],C) \simeq \TC(C[t],C) \]
where the topological cyclic homology $\TC$ is constructed from the cyclotomic structure on $\THH$. As we have seen above, $\THH(C[t])$ involves homotopy orbits of cyclic group actions, which makes it plausible that $p$-adic calculations are more complicated than the rational ones provided above. Nevertheless, Hesselholt--Madsen have computed $K(k[\epsilon],k)$ for $k$ a perfect field of characteristic $p$ \cite{HM}, and Riggenbach \cite{Riggenbach} has extended the formula to perfectoid rings, using a similar method as Speirs \cite{Speirs1} had used to reproduce Hesselholt and Madsen's calculation via the Nikolaus--Scholze approach to $\TC$ \cite{NS}. For instance, as a consequence, for $k$ a perfect field of characteristic $p$ and $n>0$, we obtain for $|t|=1$
\[ K_n(k[t],k) \cong \begin{cases} \mathbb{W}_{n}(k)/V_2 \mathbb{W}_{n/2}(k) & \text{ if $n$ is even } \\ 0 & \text{ if $n$ is odd.} \end{cases}\]
Here, $\mathbb{W}(k)$ denotes the ring of big Witt vectors over $k$ and $V$ is the Verschiebung operator.

Moreover, using again the above translation between $\Z[\epsilon]$ and $\Z[t]$ with $|t|=1$, the results of Angeltveit--Gerhardt--Hesselholt \cite{AGHess} give the following result:
\[ K_{2n+2}(\Z[t],\Z) \cong \Z \quad \text{ and } |K_{2n+1}(\Z[t],\Z)| = (2n)!\]
\end{rem}

\subsubsection*{Generalized homotopy invariance of $K$-theory}
One can also combine Theorem~\ref{thm:free-vs-sz} nicely with a recent result of Burklund--Levy to extend \cite[Theorem 3 \& 4]{Waldhausen1}, see also \cite[Example 5.6]{BL}:
\begin{cor}\label{cor:free-homotopy-invariance}
Let $C$ be a right regular coherent discrete ring and $M$ a discrete $(C,C)$-bimodule which is left $C$-flat. Then the canonical map 
\[ K(C) \lto K(T_C(M)) \]
is an equivalence on connective covers.
\end{cor}
\begin{proof}
By Theorem~\ref{thm:free-vs-sz}, we may equivalently show that the map $K(C) \to K(C \oplus \Omega M)$ is an equivalence on $(-1)$-connective covers. This follows from the d\'evissage result of Burklund--Levy \cite[Theorem 1.1]{BL}.
\end{proof}
We remark that if $C$ is right regular Noetherian, then both $K$-theories appearing in Corollary~\ref{cor:free-homotopy-invariance} are in fact connective, so that in this case, the map is an equivalence.

As a further illustration of the way one can employ Theorem~\ref{thm:free-vs-sz}, we offer the following explicit example:
\begin{ex}\label{cor:free-algebra-on-torsion-degree-1}
Let $M$ be any discrete $\Z$-module. Then the map $K(\Z) \to K(T_\Z(\Omega M))$ is an equivalence. Note that we cannot directly apply \cite[Theorem 1.1]{BL}, as $\pi_{-1}(T_{\Z}(\Omega M))$ is not  flat over $\Z$  in general (it contains $M$ as a direct summand). Instead, we appeal to Theorem~\ref{thm:free-vs-sz} and consider the following pullback square
\[
\begin{tikzcd}
	K(\Z\oplus \Omega^2M) \ar[r] \ar[d] & K(\Z) \ar[d] \\
	K(\Z) \ar[r] & K(T_\Z(\Omega M)),
\end{tikzcd}
\]
whose top horizontal map has a section induced by the connective cover map $\Z \to \Z \oplus \Omega^2 M$. As the global dimension of $\Z$ is 1, $M = \pi_{-2}(\Z \oplus M)$ has Tor-dimension $<2$ over $\Z$, and hence the conditions of \cite[Theorem 1.1]{BL} are satisfied for this map. Using that $\Z$ is Noetherian and combining \cite[Theorem 1.1]{BL} with \cite[Theorem 1.2]{AGH}, we find that the map $\Z \to \Z\oplus \Omega^{2}M$  in fact induces an equivalence on non-connective $K$-theory, not only on $(-1)$-connective $K$-theory, see also \cite[Section 3.3]{BL}.

The same argument applies verbatim to any Dedekind domain $R$ with a (possibly non-flat) module $M$ to show that $K(R) \to K(T_R(\Omega M))$ is an equivalence.
\end{ex}

In contrast to Example~\ref{cor:free-algebra-on-torsion-degree-1}, the map $K(\Z) \to K(\Z\oplus \Omega \Z/2)$ is not an equivalence. We have calculated this in a discussion with Burklund and Levy; see \cite[Example 5.13]{BL} for a similar calculation and \cite[Footnote 15]{BL}.
This shows in particular that the flatness assumption in Corollary~\ref{cor:free-homotopy-invariance} cannot generally be dropped.

\subsection{Waldhausen's generalized free products} 
We explain next how our results relate to Waldhausen's \cite{Waldhausen1, Waldhausen2}, more precisely how our results (in combination with \cite{BL}) reproduce and generalize his theorem about the $K$-theory of certain generalized free products of discrete rings. As always, let $k$ be an $\E_2$-algebra in spectra. Following Waldhausen, a map of $k$-algebras $C \to A$ is called a pure embedding, if it admits a $(C,C)$-linear retraction. For a span $A \leftarrow C \to B$ of pure embeddings of $k$-algebras, we denote 
by $\bar A$ and $\bar B$ the cofibres of the maps from $C$ to $A$ and $B$, respectively. Let us also define the $(C,C)$-bimodule $M = \bar B \otimes_C \bar A$. The following is our main result in this context.

\begin{thm}\label{cor:globale-zusammenfassung}
Let $E$ be a $k$-localizing invariant, and let $A \leftarrow C \to B$ be a span of pure embeddings of $k$-algebras.
Then there is a canonical decomposition
\[
E(A \amalg_{C} B) \simeq E(A) \oplus_{E(C)} E(B) \oplus \Sigma\Nil E(C;M).
\]
If $A$, $B$, and $C$ are discrete, and $\bar A$ and $\bar B$ are both left $C$-flat, then $A \amalg_{C} B$ is discrete as well, and equals the pushout in the category of discrete rings. If moreover the ring $C$ is right regular Noetherian (or coherent), then $\Sigma \Nil K(C;M)$ vanishes (after applying $\tau_{\geq 0}$).
\end{thm}

\begin{rem}
\begin{enumerate}
\item Theorem~\ref{cor:globale-zusammenfassung} in particular reproves \cite[Theorem 1]{Waldhausen1}, which treats connective $K$-theory\footnote{The result has already been extended to nonconnective $K$-theory in \cite{BartelsLueck}.} of discrete rings, and the part of \cite[Theorem~4]{Waldhausen1} which concerns \cite[Theorem 1]{Waldhausen1}. We remark that Waldhausen makes some stronger assumptions on the underlying left module structure of the bimodules $\bar{A}$ and $\bar{B}$.
\item One can also apply Theorem~\ref{cor:globale-zusammenfassung} to trivial square zero extensions $C \oplus \bar A \leftarrow C\to C \oplus \bar B$. Together with Corollary~\ref{cor:Nil-vs-N}, this ought to recover and is certainly related to \cite[Theorem 0.4]{DKR}. As \cite{DKR} works with ordinary categories, see again Remark~\ref{rem:exact-vs-stable-concrete-case} for the relevant comparison.
\end{enumerate}
\end{rem}

In order to prove Theorem~\ref{cor:globale-zusammenfassung} we will need a number of preliminary observations. We begin with the following. We consider a span of pure embeddings  $A \leftarrow C \to B$ and form the pushout square of $k$-algebras
\[
\begin{tikzcd}
 C \ar[d]\ar[r] & B \ar[d] \\ 
 A \ar[r] & A \amalg_{C} B.
\end{tikzcd}
\]
\begin{lemma}\label{prop:waldhausen-ring-is-sz}
The ring $\widetilde{C}$ associated by Theorem~\ref{ThmA} to the above pushout square is canonically equivalent to the trivial square zero extension $C \oplus \Omega M$. In particular, it depends only on $C$ and the $(C,C)$-bimodule $\bar B \otimes_C \bar A$.
\end{lemma}
\begin{proof}
By construction, $\widetilde{C} = A \boxtimes_{B\otimes_{C}A}B$ for the Milnor context $(A, B, B\otimes_C A)$. We may also consider the Milnor context $(C,C,C \oplus M)$ and observe that the commutative diagram
\begin{equation} \label{eq:comparison-cospan}
\begin{tikzcd}
	C \ar[d] \ar[r] & C \oplus M \ar[d] & C \ar[l] \ar[d] \\
	A \ar[r] & B\otimes_C A & B \ar[l] 
\end{tikzcd}
\end{equation}
provides a map $ (C,C,C \oplus M) \to (A,B,B\otimes_C A)$ of Milnor contexts in the sense of Definition~\ref{def:map-of-milnor-contexts}, where the middle vertical map is induced by $(C,C)$-bimodule retractions of the pure embeddings. 
By Lemma~\ref{lemma:map-of-milnor-contexts}, we get an induced map of $k$-algebras $C \boxtimes_{C\oplus M} C \to A \boxtimes_{B\otimes_{C}A}B$ whose underlying map of spectra is obtained by taking pullbacks of the horizontal lines in \eqref{eq:comparison-cospan}. By construction, this map is an equivalence, and hence so is the above map of $k$-algebras.
Finally, by Lemma~\ref{lem:pullback-bimodule}, the $k$-algebra $C \boxtimes_{C\oplus M} C$  identifies with the trivial square zero extension $C \oplus \Omega M$. 
\end{proof}

In order to check when the the pushout $A \amalg_{C} B$ is discrete provided $A,B,$ and $C$ are, we
will make use of the following well known calculation. For the reader's convenience, we provide an argument here. We thank Achim Krause for suggesting the use of gradings in the proof.
\begin{lemma}\label{lemma:tensor-over-sz}
Let $C$ be a $k$-algebra and $M$ a $(C,C)$-bimodule. There is a canonical  equivalence of $(C,C)$-bimodules
\[ 
C \otimes_{C\oplus \Omega M} C \simeq \bigoplus\limits_{n \geq 0} M^{\otimes_{C}^n}.
\]
\end{lemma}
\begin{proof}
We observe that the trivial square zero extension $C \oplus \Omega M$ is the underlying augmented algebra in $\BMod(C,C)$ of a graded algebra where we put $\Omega M$ in degree 1. We then consider  the fibre sequence of graded $(C \oplus \Omega M,C \oplus \Omega M)$-bimodules induced by the augmentation map
\[ 
\Omega M \lto C \oplus \Omega M \lto C 
\]
and observe that the $(C\oplus \Omega M, C\oplus \Omega M)$-bimodule structure on $\Omega M$ obtained from this sequence is the  $(C,C)$-bimodule structure restricted along the augmentation. Applying the functor $(-)\otimes_{C\oplus \Omega M} C$ and using the equivalence
\[ 
\Omega M \otimes_{C \oplus \Omega M} C \simeq \Omega M \otimes_C (C \otimes_{C\oplus \Omega M} C)
\]
we obtain a fibre sequence of $(C \oplus\Omega M, C)$-bimodules and, a fortiori, of $(C,C)$-bimodules
\[ 
\Omega M \otimes_C (C \otimes_{C\oplus \Omega M} C) \lto C \lto C \otimes_{C \oplus \Omega M} C
\]
whose final map admits the multiplication map as a retraction. Consequently, we have an equivalence
\[
 C \otimes_{C\oplus \Omega M} C \simeq C \oplus \big[M \otimes (C\otimes_{C\oplus \Omega M} C)\big],
 \]
where $M$ is still in degree $1$. Inductively, we see that the degree $n$ part of the left-hand side is given by $M^{\otimes^{n}_C}$. Finally, we use that the underlying object functor taking a graded object to the sum of its pieces is symmetric monoidal and commutes with colimits. This finishes the proof of the lemma.
\end{proof}

\begin{lemma}\label{lemma:pushout-discrete}
Let $A \leftarrow C \to B$ be a span of pure embeddings of discrete $k$-algebras. If $A$ and $B$ are left $C$-flat, the pushout $A \amalg_C B$ in $\Alg(k)$ is discrete, and thus equals the corresponding pushout in the category of discrete $k$-algebras.
\end{lemma}
In fact, the proof shows that the same conclusion holds if $A$ and $B$ are assumed to be right $C$-flat.
\begin{proof}
Combining Theorem~\ref{ThmB}, Theorem~\ref{thm:GLT}, and Lemma~\ref{prop:waldhausen-ring-is-sz}, it suffices to prove that $A \otimes_{C \oplus \Omega M} B$ is discrete, where $M = \bar B\otimes_C \bar A$ and $\bar A =\cof(C \to A)$ and $\bar B =\cof(C \to B)$ are as before. By diagram~\eqref{eq:comparison-cospan} above and Lemma~\ref{lemma:tensor-over-sz}, we have equivalences
\[ A \otimes_{C \oplus \Omega M} B \simeq A \otimes_C (C \otimes_{C \oplus \Omega M} C) \otimes_C B \simeq \bigoplus\limits_{n\geq 0} A \otimes_C M^{\otimes_C^n} \otimes_C B.\]
Note that the tensor product (over $C$) of two left flat $(C,C)$-bimodules is again a left flat $(C,C)$-bimodule. We deduce that $M$ and therefore $M^{\otimes_C^n}$ are left flat $(C,C)$-bimodules for all $n \geq 0$. Consequently, each summand in the final term in the above display is a left flat $(C,C)$-bimodule, and therefore so is the sum as needed. 
\end{proof}

We can now prove our main result of this subsection.
\begin{proof}[Proof of Theorem~\ref{cor:globale-zusammenfassung}]
Applying Lemma~\ref{prop:waldhausen-ring-is-sz} and Theorem~\ref{ThmA}, we obtain the following $E$-cartesian diagram.
\[ \begin{tikzcd}
	C\oplus \Omega M \ar[r] \ar[d] & B \ar[d] \\
	A \ar[r] & A \amalg_C B
\end{tikzcd}\]
By diagram \eqref{eq:comparison-cospan}, the left vertical and top horizontal map are induced by the composite of $C\oplus \Omega M \to C$ followed by the canonical maps to $A$ and $B$, respectively. As a consequence, we obtain a cartesian diagram 
\[\begin{tikzcd}
	E(C) \oplus \Nil E(C;M) \ar[r] \ar[d] & E(B) \ar[d] \\
	E(A) \ar[r] & E(A \amalg_C B)
\end{tikzcd}\]
where the left vertical and top horizontal maps again factor over the projection to $E(C)$, so we obtain the canonical decomposition as claimed. 
The second claim is Lemma~\ref{lemma:pushout-discrete}. 
Finally, the vanishing of the $K$-theoretic Nil-term  under the stated regularity assumptions follows from \cite[Theorem 1.1]{BL}.
\end{proof}

\begin{ex}\label{Ex:A1-invariant}
In this example, we explain a consequence of Theorem~\ref{cor:globale-zusammenfassung} to $\mathbb{A}^1$-invariant $k$-localizing invariants. We recall that a $k$-localizing invariant is called $\A^1$-invariant
if, for discrete $k$-algebras $C$, the canonical map $E(C) \to E(C[x])$ is an equivalence. Examples include Weibel's homotopy $K$-theory $KH(-)$, $K(1)$-localized $K$-theory $L_{K(1)}K(-)$ \cite[Corollary 4.24]{LMMT}\footnote{This result follows already from the work of Bhatt--Clausen--Mathew \cite{BCM} as pointed out in \cite{LMMT}.}, and periodic cyclic homology $\HP(-/k)$ over $\Q$-algebras $k$ \cite[Cor.~3.12]{Kassel}. Moreover, Sosnilo recently proved that if $k$ is connective with $\pi_0(k)$ an $\F_p$- (or $\Z[\tfrac{1}{p}]$)-algebra and $E$ is a finitary $k$-localizing invariant, then $E(-)[\tfrac{1}{p}]$ (or $E(-)/p$, respectively) is $\mathbb{A}^1$-invariant, extending Weibel's results \cite{Weibel1} from $K$-theory to general finitary $k$-localizing invariants, \cite{Sosnilo}.

An argument of Weibel \cite[Theorem 1.2 (iv)]{Weibel-homotopy} shows that for any such $\A^1$-invariant $k$-localizing invariant and any left flat discrete $(C,C)$-bimodule $M$, the canonical map $E(C) \to E(T_C(M))$ is an equivalence, so that $NE(C;M) = 0$: Indeed, in loc.\ cit.\ it is observed that if $A = \oplus_{n\geq 0} A_n$ is a non-negatively graded discrete ring, then there is a canonical map $A \to A[t]$ sending $(a_n)_{n\geq 0}$ to $(a_n t^n)_{n\geq 0}$. Evaluating this map on $t=0$ and $t=1$ we find that the identity of $A$ is $\mathbb{A}^1$-homotopic to the composite $A \to A_0 \to A$, giving that the canonical map $E(A_0) \to E(A)$ is an equivalence. Since the tensor algebra $T_C(M)$ is a non-negatively graded discrete ring (since $M$ is left flat) with graded degree 0 part equal to $C$, the claim follows.

As a consequence of Theorem~\ref{cor:globale-zusammenfassung} we then obtain the following result for $\A^1$-invariant $k$-localizing invariants. Namely, for pure embeddings $A \leftarrow C \to B$ of discrete rings such that $A$ and $B$ are  flat as left $C$-modules, the square
\[\begin{tikzcd}
	E(C) \ar[r] \ar[d] & E(B) \ar[d] \\
	E(A) \ar[r] & E(A \amalg_C B)
\end{tikzcd}\]
is cartesian. Indeed, by Corollary~\ref{cor:free-homotopy-invariance}, the failure for this square being cartesian is equivalant to $NE(C;M)$ which vanishes as we have explained above.

For homotopy $K$-theory this result was proven earlier by Bartels and L\"uck in \cite[Theorem 11.3]{BartelsLueck} and for periodic cyclic homology over a $\Q$-algebra, this generalizes a calculation of Burghelea \cite[Prop.~$\mathrm{II}_p$]{Burghelea} from free to more general amalgamated products, see Example~\ref{Ex:Burghelea} for the details.
\end{ex}

\subsubsection*{Group rings}
A special class of spans $A \leftarrow C \to B$ of pure embeddings of discrete rings (even satisfying Waldhausen's stronger assumptions on $\bar A$ and $\bar B$ indicated above) is given by embeddings of group rings induced by the inclusion of a common subgroup $H$ in two groups $G$ and $G'$. 
Concretely, let $R$ be a $k$-algebra, and consider the span $R[G] \leftarrow R[H] \to R[G']$, and let $M$ be the corresponding $(R[H], R[H])$-bimodule considered above. We note that there is a canonical equivalence
\[ R[G] \amalg_{R[H]} R[G'] \simeq R[G \star_{H} G'],\]
where $G \star_H G'$ denotes the amalgamated free product. 
Indeed, we have $R[G] = R\otimes_\SS \SS[G]$, the functor $\Alg(\SS) \to \Alg(R)$ sending $A$ to $R\otimes_\SS A$ preserves colimits, and the forgetful functor $\Alg(R) \to \Alg(\SS)$ preserves contractible colimits.
The latter can be checked after passing to categories of perfect modules (see the first paragraph of the proof of Theorem~\ref{thm:thm-B-im-Text}) and using that the forgetful functor $\Cat^{R}_{\infty} = \Mod_{\Perf(R)}(\Cat_{\infty}^{\perf}) \to \Cat_{\infty}^{\perf}$ preserves all colimits.
Moreover the functor $\SS[-]\colon \Grp(\Spc) \to \Alg(\Sp)$ is a left adjoint and hence also preserves colimits. Finally, it is a classical computation that the pushout in $\Grp(\Spc)$ of discrete groups along inclusions is a discrete group and hence coincides with the pushout in $\Grp(\Set)$, see \cite[Theorem 7.1.9]{Geoghegan} for a generalization of this fact and the proof of Proposition~\ref{prop:coordinate-axes-sphere} for details.

\begin{cor}\label{cor:waldhausen-group-ring}
In the above situation, we have 
\[E(R[G \star_H G']) \simeq E(R[G]) \oplus_{E(R[H])} E(R[G]) \oplus \Sigma \Nil E(R[H]; M)\] 
where $M$ is the $(R[H],R[H])$-bimodule as above.
\end{cor}

\begin{ex}
As an amusing special case of the the above, we obtain the following result which for algebraic $K$-theory and in the case of discrete rings was already proven in \cite{Chen:2012uy} and \cite[Corollary 3.27]{DKR}. 
So let $R$ be a $k$-algebra. Let $C_{2}$ denote the cyclic group of order 2. Then $C_{2} \star C_{2}$ is isomorphic to the infinite dihedral group $D_{\infty}$.
Then the failure of the square
\[ \begin{tikzcd}
	E(R) \ar[r] \ar[d] & E(R[C_2]) \ar[d] \\
	E(R[C_2]) \ar[r] & E(R[D_\infty])
\end{tikzcd}\]
being a pushout is equivalent to $NE(R;R)$. Indeed, in this situation the $(R,R)$-bimodule $M$ is just $R$, so the claim follows from Corollary~\ref{cor:Nil-vs-N}.

In fact, more generally, for $H \subseteq G$ a central subgroup of index 2, one can consider the amalgamated product $G \star_H G$. In this case, one again finds that the relevant $(R[H],R[H])$-bimodule $M$ is simply $R[H]$, therefore $NE(R[H];R[H])$ is the failure of the diagram
\[\begin{tikzcd}
	E(R[H]) \ar[r] \ar[d] & E(R[G) \ar[d] \\
	E(R[G]) \ar[r] & E(R[G\star_H G])
\end{tikzcd}\]
being a pushout.
\end{ex}

\begin{ex}\label{Ex:Burghelea}
In this example, we explain a relation to the work of Burghelea \cite{Burghelea} where (periodic) cyclic homology of group rings is studied. For this, we let $k$ be a commutative ring and, for $k$-algebras $A$, we denote by $\HC(A) = \HH(A/k)_{h\T}$ and $\HP(A) = \HH(A/k)^{t\T}$ the cyclic and periodic cyclic homology of $A$ relative to $k$, and by $\widetilde{\HC}(A)$ the cofibre of the canonical map $\HC(k) \to \HC(A)$, and similarly for $\HP$.
Now let $G$ and $G'$ be discrete groups. \cite[Prop.~II \& $\mathrm{II}_p$]{Burghelea} say the following, where $\Lambda$ is the set of conjugacy classes of elements of $G \star G'$ disjoint from the subgroups $G$ and $G'$:
\begin{enumerate}
\item $\widetilde{\HC}(k[G \star G']) = \widetilde{\HC}(k[G]) \oplus \widetilde{\HC}(k[G']) \oplus \bigoplus\limits_{\Lambda} k$, and
\item $\widetilde{\HP}(k[G\star G']) = \widetilde{\HP}(k[G]) \oplus \widetilde{\HP}(k[G'])$,
\end{enumerate}
though in loc.\ cit.\ it is not explicitly written that reduced cyclic and periodic homologies are considered.
We show here that this result is correct precisely for $\Q$-algebras (the case (2) for $\Q$-algebras follows e.g.\ from Example~\ref{Ex:A1-invariant} above) and give a corrected formula for the (periodic) cyclic homology of $k[G\star G']$ for general commutative rings $k$. Burghelea has also published an erratum to his paper on the arxiv \cite{Burghelea-erratum}, see specifically Proposition 3.2 there. Specializing Corollary~\ref{cor:waldhausen-group-ring} to (periodic) cyclic homology relative to $k$ and using Corollary~\ref{cor:Nil-vs-N}, we obtain the following result:
\begin{align*}
\widetilde{\HC}(k[G\star G']) \simeq \widetilde{\HC}(k[G]) \oplus \widetilde{\HC}(k[G']) \oplus N\HC(k;M) \\
\widetilde{\HP}(k[G\star G']) \simeq \widetilde{\HP}(k[G]) \oplus \widetilde{\HP}(k[G']) \oplus N\HP(k;M)
\end{align*}
where $M$, as one directly checks, is the free $k$-module on the set $S= (G\setminus\{e\}) \times (G' \setminus \{e'\})$. 
Moreover, we have 
\[ N\HC(k;M) = \widetilde{\HH}(T_k(M))_{h\T} \quad \text{ and } N\HP(k;M) = \widetilde{\HH}(T_k(M))^{t\T}.\]
Now, \cite[Prop.~3.1.5]{Hesselholt-Witt} says that reduced Hochschild homology of tensor algebras decomposes $\T$-equivariantly as 
\[ 
\widetilde{\HH}(T_k(M)) = \bigoplus\limits_{\omega \in \Omega_0(S)} k \otimes (S^1/C(\omega))_+,
\]
where $\Omega_0(S)$ is the set of non-empty cyclic words in $S$, and $C(\omega)$ is a cyclic group of finite order $n(\omega)$ associated with the word $\omega$.
Applying homotopy $\T$-orbits, we obtain the following equivalence, see Lemma~\ref{lemma:convenience} for a generalization which we will use in Example~\ref{ex:amalgamation-with-central-subgroup}.
\[ 
N\HC(k;M) \simeq \widetilde{\HH}(T_k(M))_{h\T} \simeq \bigoplus\limits_{\omega \in \Omega_0(S)}  k \otimes BC(\omega)_+
\]
From the norm fibre sequence 
\[\Sigma\widetilde{\HH}(T_k(M))_{h\T} \lto \widetilde{\HH}(T_k(M))^{h\T} \lto \widetilde{\HH}(T_k(M))^{t\T}\]
the fact that the homotopy $\T$-fixed points are 1-coconnective, the Tate construction is 2-periodic, and the group homology of $C(\omega)$ is given by the kernel and cokernel of the multiplication by $n(\omega)$ on $k$ in positive even and odd degrees, respectively, we then find  
\[ N\HP(k;M) \simeq \widetilde{\HH}(T_k(M))^{t\T} \simeq \bigoplus\limits_{\omega \in \Omega_0(S)} \Sigma k/n(\omega)[u^{\pm}] \]
where $|u|=-2$ and $k/n(\omega)$ denotes the cofibre of multiplication by $n(\omega)$ on $k$.
In total we obtain
\begin{align*}
\widetilde{\HC}(k[G\star G']) \simeq \widetilde{\HC}(k[G]) \oplus \widetilde{\HC}(k[G']) \oplus  \bigoplus\limits_{\omega \in \Omega_0(S)}  k \otimes BC(\omega)_+ \\
\widetilde{\HP}(k[G\star G']) \simeq \widetilde{\HP}(k[G]) \oplus \widetilde{\HP}(k[G']) \oplus  \bigoplus\limits_{\omega \in \Omega_0(S)}  \Sigma k/n(\omega)[u^{\pm}]
\end{align*}
In case $k$ contains $\Q$, we find that $k/n(\omega) = 0$ and $k \otimes BC(\omega)_+ \simeq k$. Moreover we note that there is a canonical bijection between the indexing sets $\Lambda$ and $\Omega_0(S)$ appearing above, as follows from the general description of conjugacy classes of elements in free products of groups as cyclically reduced cyclic words in the alphabet given by the two groups. Consequently, \cite[Prop.~II \& $\mathrm{II}_p$]{Burghelea} are correct if $k$ contains $\Q$. If not, then already the simplest non-trivial case gives a counter example: In case $G=G' = C_2$ is the cyclic group of order 2, the above concretely become
\[ 
N\HC(k;k) \simeq \bigoplus\limits_{m \geq 1}  k \otimes  (BC_m)_+ \quad \text{ and } \quad N\HP(k;k) \simeq \bigoplus\limits_{m \geq 1} \Sigma k/m[u^\pm].
\]
Therefore, if $k$ does not contain $\Q$, $N\HC(k;k)$ and $N\HP(k;k)$ have non-trivial terms of non-zero degree. 
\end{ex}

\begin{ex}\label{ex:amalgamation-with-central-subgroup}
We now expand on Example~\ref{Ex:Burghelea} and consider the case of a (possibly non-trivial) central subgroup $H \subseteq G,G'$. As above, we fix a commutative ring $k$ and aim to evaluate Hochschild and cyclic homology of $k[G\star_H G']$. In this situation, it turns out that the $(k[H],k[H])$-bimodule $M$ appearing above is the (symmetric) bimodule associated to the free $k[H]$-module on the set $S=(G/H \setminus eH) \times (G'/H \setminus e'H)$. In particular, we find that $T_{k[H]}(M) = T_k(M') \otimes_k k[H]$, where $M'$ is the (symmetric) bimodule associated to the free $k$-module on the set $S$. Applying Corollary~\ref{cor:waldhausen-group-ring} and Corollary~\ref{cor:Nil-vs-N} to the $k$-localizing invariant $\HH \colon \Cat_\infty^k \to \Mod_k^{B\T}$ then gives the following equivalence of $k$-modules with $\T$-action:
\[ \HH(k[G \star_H G']) \simeq \big( \HH(k[G]) \oplus_{\HH(k[H])} \HH(k[G']) \big) \oplus N\HH(T_{k[H]}(M)).\]
We now aim to describe the error term $N\HH(T_{k[H]}(M))$ and its homotopy $\T$-orbits in more detail.
First, since $\HH$ is a symmetric monoidal functor, we have an equivalence
\[ \HH(T_{k[H]}(M)) \simeq \HH(T_k(M') \otimes_k k[H]) \simeq \HH(T_k(M')) \otimes_k \HH(k[H])\]
and the former term is given as in Example~\ref{Ex:Burghelea}. Therefore, in total we obtain an equivalence
\[ N\HH(T_{k[H]}(M)) \simeq \bigoplus\limits_{\omega \in \Omega_0(S)} (S^1/C(\omega))_+ \otimes_\SS \HH(k[H]).\]
Applying homotopy $\T$-orbits we obtain by Lemma~\ref{lemma:convenience}
\[ N\HC(T_{k[H]}(M)) \simeq \bigoplus\limits_{\omega \in \Omega_0(S)} \HH(k[H])_{hC(\omega)},\]
generalizing the case of Example~\ref{Ex:Burghelea} where $H= \{e\}$ is the trivial group, in which case we have $\HH(k)_{hC(\omega)} \simeq k \otimes BC(\omega)_+$.
To describe the homotopy $C(\omega)$-orbits appearing above more concretely, we recall that for any discrete group $G$, we have 
\[\HH(k[G]) \simeq k \otimes L(BG)_+ \simeq \bigoplus\limits_{\langle g \rangle \in G^c/G} k \otimes (BC_g(G))_+ \]
where $G^c/G$ denotes the set of conjugacy classes of $G$,  $C_g(G)$ denotes the centralizer of $g$ in $G$ and $L(-)$ is the free loop space functor. The latter equivalence is $\T$-equivariant when the right hand side is equipped with a diagonal $\T$-action whose action on $BC_g(G)$ is given as follows, see e.g.\ \cite{Burghelea}: The inclusion $\Z/\mathrm{ord}(g)\Z \to C_g(G)$ determined by $g$ induces a fibre sequence
\[ BC_g(G) \lto B(C_g(G)/(\Z/\mathrm{ord}(g)\Z) \lto B^2(\Z/\mathrm{ord}(g)\Z) \]
which can be pulled back along the canonical map $B\T \to B^2(\Z/\mathrm{ord}(g)\Z)$ (which  is an equivalence if the order of $g$ is infinite). The resulting fibre sequence then describes the relevant $\T$-action on $BC_g(G)$.
Consequently, we find 
\[ N\HC(T_{k[H]}(M)) \simeq \bigoplus\limits_{\omega \in \Omega_0(S)} \bigoplus\limits_{\langle h \rangle \in H^c/H} k \otimes [(BC_h(H))_{hC(\omega)}]_+ \]
so we shall content ourselves with describing $(BC_h(H))_{hC_n}$ for natural numbers $n\geq 1$.
From the above, we find that
\[ (BC_h(H))_{hC_n} \simeq B(C_h(H)/(\Z/\mathrm{ord}(h)\Z)) \times_{B^2(\Z/\mathrm{ord}(h)\Z)} BC_n \]
where the pullback is formed over the canonical composite $BC_n \to B\T \to B^2(\Z/\mathrm{ord}(h)\Z)$.
The homology of these spaces with coefficients in $k$ can then in principle be described via Serre spectral sequences. For instance, this is easy to work out  in the case $H=\Z$, where one finds the following:
\[ \HH(k[\Z])_{hC(\omega)} \simeq \bigoplus\limits_{\ell\in \Z} (BC_\ell \times_{B\T} BC(\omega))_+ \otimes k \]
where our convention here is that $BC_\ell = \fib(\ell \colon B\T \to B\T)$ also for $\ell = 0$. The homology of the spaces $BC_\ell \times_{B\T} BC(\omega)$ is then immediate from the observation that the map to $BC(\omega)$ is a $\T$-bundle classified by the element $\ell x \in H^2(BC(\omega);\Z)$, and the Serre spectral sequence. 
We finally note that $\HH(k[\Z])$ is 1-coconnective, so that $N\HP(T_{k[\Z]}(M))$ consequently agrees with $N\HC(T_{k[\Z]}(M))$ in sufficiently high degrees and is 2-periodic.
\end{ex}

\begin{ex}
As a concrete example of the above, we briefly discuss the case of $\mathrm{SL}_2(\Z)$, which is isomorphic to the amalgamated product $C_4 \star_{C_2} C_6$. As worked out in Example~\ref{ex:amalgamation-with-central-subgroup}, we then obtain
\[ \HH(k[\mathrm{SL}_2(\Z)]) \simeq \big( \HH(k[C_4]) \oplus_{\HH(k[C_2])} \HH(k[C_6])\big)\oplus \bigoplus\limits_{\omega \in \Omega_0(S)} (S^1/C(\omega))_+ \otimes_\SS \HH(k[C_2])\]
where $S$ is a set with two elements. In order to evaluate cyclic homology concretely, we recall that we have
\[ \HC(k[\mathrm{SL}_2(\Z)]) \simeq \big( \HC(k[C_4]) \oplus_{\HC(k[C_2])} \HC(k[C_6])\big)\oplus \bigoplus\limits_{\omega \in \Omega_0(S)}  \HH(k[C_2])_{hC(\omega)}\]
and that as described in Example~\ref{ex:amalgamation-with-central-subgroup} we have
\[ \HH(k[C_2])_{hC(\omega)} \simeq \big[(BC_2 \times BC(\omega))_+ \otimes k\big] \oplus \big[(BC_2 \times_{B\T} BC(\omega))_+ \otimes k\big]\]
where the final pullback is along the two canonical maps $BC_2 \to B\T \leftarrow BC(\omega)$.
\end{ex}

Finally, for the reader's convenience, we record the following well-known lemma which we have used above.
\begin{lemma}\label{lemma:convenience}
There is a canonical equivalence between the following two functors $\Mod(k)^{B\T} \to \Mod(k)$:
\[ ((S^1/C_n)_+ \otimes -)_{h\T}  \simeq (-)_{hC_n}.\]
\end{lemma}
\begin{proof}
Recall that for a map of spaces $f\colon X \to Y$ and any cocomplete $\infty$-category $\cC$, the pullback functor $f^*\colon \Fun(Y,\cC) \to \Fun(X,\cC)$ has a left adjoint $f_!$ given by left Kan extension. 
Let us then consider the following pullback diagram of spaces
\[\begin{tikzcd}
	BC_n \ar[r,"s"] \ar[d,"j"] & \star \ar[d,"i"] \\
	B\T \ar[r,"n"] & B\T
\end{tikzcd}\]
where $n$ denotes the multiplication by $n$ map. Let us also denote by $r \colon B\T \to \star$ the unique map. With these notations, we have $(-)_{h\T} = r_!$ and $(-)_{hC_n} = s_!$. We now note that $S^1$ as a $\T$-space via left multiplication is given by $i_!(\ast)$ where $\ast \in \Fun(\star,\Spc) \simeq \Spc$ is a terminal object. Likewise, the $\T$-space $S^1/C_n$ is given by $n^*i_!(\ast)$. Since the above diagram is a pullback, the canonical Beck--Chevalley map $n^*i_!(\ast) \to j_!(s^*(\ast)) = j_!(\ast)$ is an equivalence. 
We then calculate for $M \in \Mod(k)^{B\T}$
\[ r_!(n^*i_!(\ast) \otimes M) \simeq r_!(j_!(\ast) \otimes M) \simeq r_!(j_!(\ast \otimes j^*(M)) = s_!(j^*(M))\]
where the second equivalence is the projection formula for $(j_!,j^*)$.
The left and right most terms are the functors under investigation as we have explained above, so the lemma is proven.
\end{proof}

\subsection{Generalized coordinate axes}
We begin with a discussion of the coordinate axes over the sphere spectrum. In more detail, we consider the following pullback square of $\E_\infty$-rings.
\[\begin{tikzcd}	
	\SS[x,y]/(xy) \ar[r] \ar[d] & \SS[y] \ar[d] \\
	\SS[x] \ar[r] & \SS
\end{tikzcd}\]

\begin{prop}\label{prop:coordinate-axes-sphere}
The $\wtimes{}{}$-ring associated to the above pullback diagram is given by $\SS[t]$ where $|t|=2$.
\end{prop}
\begin{proof}
We consider the span $\SS[x] \leftarrow \SS[x,y] \to \SS[y]$ and note the equivalence
\[ \SS[x] \otimes_{\SS[x,y]} \SS[y] \simeq \SS.\]
This shows that $\SS[x,y]$ is a tensorizer for the Milnor context $(\SS[x],\SS[y],\SS)$.
By Theorem~\ref{ThmB} we therefore obtain an equivalence
\[ \SS[x] \wtimes{\SS[x,y]/(xy)}{\SS} \SS[y] \simeq \SS[x] \amalg_{\SS[x,y]} \SS[y].\]
Using the equivalence $\SS[x,y] \simeq \SS[x] \otimes \SS[y]$, we deduce that there is an $\E_1$-equivalence $\SS[x,y] \to \SS[x,y]$ sending $x$ to $x+1$ and $y$ to $y+1$. After composing with this equivalence, the map $\SS[x,y] \to \SS[x]$ becomes the map induced by the projection $p_1\colon \N \times \N \to \N$, similarly for the map $\SS[x,y] \to \SS[y]$.\footnote{This step was omitted in the previous version. Thanks to Andrei Konovalov for pointing this out.} Hence, the pushout on the right hand side is given by $\SS[\N \amalg_{\N\times \N} \N]$, since the functor 
\[\SS[-] \colon \Mon_{\E_1}(\Spc) \to \Alg_{\E_1}(\Sp)\] 
is a left adjoint and hence commutes with colimits. In general, pushouts in $\Mon_{\E_1}(\Spc)$ are not so easy to compute, but in our case, we observe that the pushout is connected: Its $\pi_0$ coincides with the pushout of $\N \leftarrow \N \times \N \to \N$ in $\Mon(\Set)$, which is easily seen to be trivial. Since the group completion functor $\Mon_{\E_1}(\Spc) \to \Grp_{\E_1}(\Spc)$ is a left adjoint, it commutes with pushouts, so we may equivalently calculate the pushout of $\Z \leftarrow \Z \times \Z \to \Z$ in $\Grp_{\E_1}(\Spc)$. Now, the adjunction
\[ B \colon \Grp_{\E_1}(\Spc) \rightleftarrows \Spc_*^{\geq 1}\colon \Omega \]
is an adjoint equivalence. Therefore the pushout we aim to calculate is given by the loop space of the pushout of the span $S^1 \leftarrow T^2 \to S^1$. This pushout is well-known to be the join $S^1 \star S^1$ which is $S^3$. In total, we find the equivalence 
\[ \SS[x] \amalg_{\SS[x,y]} \SS[y] \simeq \SS[\Omega S^3] = \SS[t] \quad \text{ with } |t|=2\]
as claimed.
\end{proof}

\begin{cor}\label{cor:coordinate-axes}
Let $C$ be any ring spectrum and let $|t|=2$. Then there is a motivic pullback square 
\[\begin{tikzcd}
	C[x,y]/(xy) \ar[r] \ar[d] & C[y] \ar[d] \\
	C[x] \ar[r] & C[t].
\end{tikzcd}\]
\end{cor}
\begin{proof}
This follows from Propositions~\ref{prop:coordinate-axes-sphere} and \ref{prop:tensor-product}.
\end{proof}

\begin{ex}
Let us discuss Corollary~\ref{cor:coordinate-axes} in a number of examples.
First, we remind the reader of the rational calculation of Proposition~\ref{eq:K-theory-of-homotopical-polynomial-algebras}. By the above, it provides a description of the $K$-theory of the coordinate axes over a commutative regular Noetherian $\Q$-algebra in terms of differential forms and thereby reproduces computations of Geller, Reid, and Weibel \cite{GRW}.

For commutative regular Noetherian $\F_p$-algebras, Hesselholt calculated the $K$-theory of the coordinate axes in \cite{Hesselholt} in terms of big de Rham--Witt forms. Later, Speirs \cite{Speirs2} gave a new proof in the case of perfect fields $k$ of characteristic $p$ using the Nikolaus--Scholze formula for TC. Concretely, one obtains for $n\geq 0$ that
\[ K_{2n+1}(k[t],k) \cong \mathbb{W}_n(k) \quad \text{ and } \quad K_{2n}(k[t],k) = 0 \]
where $\mathbb{W}(k)$ again denotes the ring of big Witt vectors.
We mention also that Bay\i nd\i r and Moulinos calculated $K(\THH(\F_p))$ in \cite{BM}. Note that $\THH(\F_p) \simeq \F_p[t]$ by B\"okstedt periodicity. Corollary~\ref{cor:coordinate-axes} is then a conceptual explanation for the (at the time) purely computational fact that the $p$-adic $K$-groups of $\F_p[x,y]/(xy)$ and that of $\F_p[t]$ (recall that $|t|=2$) agree up to a shift. It is tempting to think that the $K$-theory pullback obtained via Corollary~\ref{cor:coordinate-axes} is multiplicative, which would show that the multiplication on $K_*(C[x,y]/(xy)),C)$ is trivial, but see Proposition~\ref{prop:multiplicative-structures} and Example~\ref{ex:coordinate-axes-multiplication}.

Finally, Corollary~\ref{cor:coordinate-axes} together with the calculations of Angeltveit and Gerhardt \cite{AG} give the following result for $n\geq 0$:
\[ K_{2n+1}(\Z[t],\Z) \cong \Z \quad \text{ and } \quad |K_{2n+2}(\Z[t],\Z)| = (n!)^2.\]
We note here that the pattern suggested by this and the case $|t|=1$, i.e.\ that every second relative $K$-group is free of rank $1$ does not persist to $|t|>2$ as follows already from Waldhausen \cite[Prop.\ 1.1]{Waldhausen} or \cite[Lemma 2.4]{LT} or Proposition~\ref{eq:K-theory-of-homotopical-polynomial-algebras}. 
\end{ex}

Finally, let us consider a commutative discrete ring $R$, equipped with elements $x$ and $y$ so that $R$ is an $\SS[x,y]$-algebra in the evident way. Let us write $R/x$ for the cofibre of multiplication by $x$ on $R$ and $R/x,y$ for the iterated cofibre $(R/x)/y$. For the following, we note that Proposition~\ref{prop:coordinate-axes-sphere} specifies an $\SS[x,y]$-algebra structure on $\SS[t]$.
\begin{lemma}\label{cor:generalized-coordinate-axes}
For a commutative ring $R$ with non-zero divisors $x$ and $y$, the $\wtimes{}{}$-ring for the pullback square
\[
\begin{tikzcd}
	R/xy \ar[r] \ar[d] & R/x \ar[d] \\
	R/y \ar[r] & R / x,y
\end{tikzcd}
\]
is given by $\SS[t] \otimes_{\SS[x,y]} R$.
If $(x,y)$ forms a regular sequence in $R$, then there is a canonical isomorphism $\pi_*(\SS[t] \otimes_{\SS[x,y]} R) \cong (R/x,y)[t]$.
\end{lemma}
\begin{proof}
The first claim follows from the general base-change result in Proposition~\ref{prop:base-change} with $k \to \ell$ being $\SS[x,y] \to R$. To see the second, we note the equivalence
\[ \SS[t]\otimes_{\SS[x,y]} R \simeq \SS[t] \otimes_{\SS[x,y]} \Z[x,y] \otimes_{\Z[x,y]} R \simeq \Z[t] \otimes_{\Z[x,y]} R\]
and consider the Tor spectral sequence
\[ 
\Tor_r^{\Z[x,y]}(\pi_*(\Z[t]), R)_{s} \Longrightarrow \pi_{r+s}(\Z[t] \otimes_{\Z[x,y]} R) \cong \pi_{r+s}(\SS[t] \otimes_{\SS[x,y]} R).
\]
The assumption that $(x,y)$ forms a regular sequence implies that the spectral sequence lives on the $r=0$ axis. Now, the homotopy ring we intend to calculate is an $(R/x,y)$-algebra and it has a filtration whose associated graded is a free algebra on a degree 2 generator. We can lift such a generator to an element in $\pi_2(\SS[t]\otimes_{\SS[x,y]} R)$ and then find that the induced map $(R/x,y)[t] \to \pi_*(\SS[t]\otimes_{\SS[x,y]} R)$ is an equivalence.
\end{proof}

\begin{ex}\label{ex:arithmetic-coordinate-axes}
Let $R= \Z[u]$ with the elements $u$ and $u+p$. Then we obtain the Milnor square for the ``arithmetic coordinate axes''
\[\begin{tikzcd}
	\Z \times_{\F_p} \Z \ar[r] \ar[d] & \Z \ar[d] \\
	\Z \ar[r] & \F_p
\end{tikzcd}\]
and the $\wtimes{}{}$-ring is given by $\Z\sslash p$, i.e.\ the $\E_1$-algebra obtained from $\Z$ by freely setting $p=0$ in $\Alg(\Z)$. We explain in Remark~\ref{rem:formality} below, that unless $p=2$, this ring spectrum is not formal, i.e.\ not equivalent to $\F_p[t]$.
\end{ex}

\begin{ex}\label{ex:Rim-square}
Let $R= \Z[u]$ equipped with the elements $u-1$ and $1+u+\dots+u^{p-1}$. Then we obtain the classical Rim square
	\[\begin{tikzcd}
		\Z[C_p] \ar[r] \ar[d] & \Z[\zeta_p] \ar[d] \\
		\Z \ar[r] & \F_p
	\end{tikzcd}\]
and the $\wtimes{}{}$-ring is given by $\Z[\zeta_p]\sslash (\zeta_p-1)$. Krause and Nikolaus have independently shown (unpublished) that there exists a canonical motivic pullback square
\[\begin{tikzcd}
	\Z[C_p] \ar[r] \ar[d] & \Z[\zeta_p] \ar[d] \\
	\Z \ar[r] & \tau_{\geq 0}(\Z^{tC_p})
\end{tikzcd}\]
using the stable module category of $C_p$ relative to $\Z$. Comparing their construction with ours, one finds that $\Z[\zeta_p] \sslash (\zeta_p-1)$ is in fact equivalent to $\tau_{\geq 0}(\Z^{tC_p})$ (as an algebra over $\Z$) and that theirs and our motivic pullback squares agree via this equivalence. It looks conceivable to us that the equivalence $\Z[\zeta_p]\sslash (\zeta_p-1) \simeq \tau_{\geq0}(\Z^{tC_p})$ can also be proven by inspecting a dga model for endomorphism rings in stable module categories as in \cite[Section 6]{Schwede}, and comparing such models to the canonical dga model of $\Z[\zeta_p]\sslash (\zeta_p-1)$. The comparison to the construction of Krause--Nikolaus, however, circumvents any such model dependent arguments. Finally, we note that as an $\E_1$-algebra over $\SS$, there is an equivalence $\tau_{\geq 0}(\Z^{tC_p}) \simeq \F_p[t]$ as follows for instance from the facts that $\tau_{\geq 0}(\Z^{tC_p})$ is $\E_2$ and $\F_p$ is the initial $\E_2$-algebra with $p=0$ by the Hopkins--Mahowald theorem, see e.g.\ \cite[Theorem 5.1]{ACB}, and that $\pi_*(\tau_{\geq 0}(\Z^{tC_p})) \cong \F_p[t]$. 
In total, one obtains a motivic pullback square
\[ \begin{tikzcd}
	\Z[C_p] \ar[r] \ar[d] & \Z[\zeta_p] \ar[d] \\
	\Z \ar[r] & \F_p[t] 
\end{tikzcd}\]
giving some new insight into the $p$-adic $K$-groups of $K(\Z[C_p])$ as $K(\F_p[t])$ is fully understood, see the discussion following Corollary~\ref{cor:coordinate-axes}. 
\end{ex}

\begin{rem}\label{rem:formality}
In examples \ref{ex:arithmetic-coordinate-axes} and \ref{ex:Rim-square} above, the elements $x$ and $y$ form a regular sequence. In particular, the homotopy ring of $\wtimes{}{}$ is isomorphic to $\F_p[t]$ in both cases. As indicated in Example~\ref{ex:Rim-square}, the ring $\Z[\zeta_p]\sslash (\zeta_p-1)$ is indeed formal (over the sphere), i.e.\ there is an equivalence of $\E_1$-ring spectra $\Z[\zeta_p] \sslash (\zeta_p-1) \simeq \F_p[t]$. For $p=2$, this in particular gives an equivalence $\Z\sslash 2 \simeq \F_2[t]$. However, for odd primes $p$, $\Z\sslash p$ is \emph{not} equivalent to $\F_p[t]$ as a ring spectrum. Indeed, Davis--Frank--Patchkoria \cite{DFP} calculated that $p$ is non-zero on $\HH(\Z\sslash p)$, in fact they show $\HH_{2n}(\Z\sslash p) \cong \Z/p^2$ for $n>0$. Consequently, by the equivalence
\[ 
\THH(\Z\sslash p) \otimes_{\THH(\Z)} \Z \simeq \HH(\Z\sslash p) 
\]
the prime $p$ must also be non-trivial on $\THH(\Z\sslash p)$. As $p$ is trivial on $\THH(\F_p[t])$, we conclude that $\Z\sslash p$ and $\F_p[t]$ cannot be equivalent as ring spectra.

It would be interesting to find a general reasoning why for certain regular sequences $(x,y)$ in $R$, the ring spectrum $\SS[t] \otimes_{\SS[x,y]} R$ is formal, i.e.\ equivalent as ring spectrum to $(R/x,y)[t]$. 
\end{rem}

\begin{ex}\label{ex:Z/p2}
Consider $R = \Z$ with the elements $x=y=p$. Then we obtain the pullback square
\[\begin{tikzcd}
	\Z/p^2 \ar[r] \ar[d] & \F_p \ar[d] \\
	\F_p \ar[r] & \F_p \otimes_\Z \F_p
\end{tikzcd}\]
where the lower tensor product is a derived tensor product. The $\wtimes{}{}$-ring  in this case is then given by $\F_p \amalg_\Z \F_p$ and has homotopy ring $\F_p[t,\epsilon]$ with $|t|=2$ and $|\epsilon|=1$. Note that by  recent work of Antieau--Krause--Nikolaus, see \cite{AKN}, much is known about the $p$-adic $K$-theory of $\Z/p^2$ -- for instance, in high enough and even degrees, it vanishes. As a result, one understands much about the $p$-adic $K$-theory of $\F_p \amalg_\Z \F_p$ -- for instance, it vanishes in high enough and odd degrees.
\end{ex}

\bibliographystyle{amsalpha}
\bibliography{cochains}

\end{document}